\documentclass[reqno,12pt,a4paper]{amsart}

\usepackage{amssymb,amsmath,amstext,amsthm,amsfonts,amscd,bbold,wasysym}
\usepackage{euscript,a4wide}
\usepackage{graphicx}
\usepackage{color,enumerate}
\usepackage[shortlabels]{enumitem} 
\usepackage[backref]{hyperref}  

\theoremstyle{plain}
\newtheorem{mainthm}{Theorem}

\newtheorem{mainclly}[mainthm]{Corollary}
\newtheorem{theorem}{Theorem}[section]

\newtheorem{corollary}[theorem]{Corollary}
\newtheorem{lemma}[theorem]{Lemma}
\newtheorem{conj}{Conjecture}

\newtheorem{definition}[theorem]{Definition}
\newtheorem{example}[theorem]{Example}

\newtheorem{remark}[theorem]{Remark}

\newtheorem{claim}{Claim}

\newcommand{\RR}{\mathbb{R}}

\newcommand{\DD}{\mathbb{D}}
\newcommand{\sS}{\mathbb{S}}
\newcommand{\TT}{\mathbb{T}}
\newcommand{\ZZ}{\mathbb{Z}}

\newcommand{\qand}{\quad\text{and}\quad}

\newcommand{\vfi}{\varphi}

\newcommand{\Mundo}{\mathfrak{X}^{1}(M)}

\newcommand{\fX}{{\mathfrak{X}}}

\newcommand{\ov}{\overline}

\newcommand{\SA}{{\mathcal A}}

\newcommand{\cC}{{\mathcal C}}

\newcommand{\cF}{{\mathcal F}}

\newcommand{\SO}{{\mathcal O}}

\renewcommand{\SS}{{\mathcal S}}

\newcommand{\SU}{{\mathcal U}}

\newcommand{\wt}{\widetilde}

\newcommand{\F}{\EuScript{F}}

\newcommand{\wh}{\widehat}

\DeclareMathOperator{\diver}{div}

\newcommand{\close}{\operatorname{Closure}}
\newcommand{\diag}{\operatorname{diag}}

\newcommand{\Leb}{\operatorname{Leb}}
\newcommand{\leb}{\operatorname{Leb}}
\newcommand{\m}{\operatorname{Leb}}
\newcommand{\closure}{\operatorname{clos}}

\newcommand{\supp}{\operatorname{supp}}
\newcommand{\crit}{\operatorname{Crit}}
\newcommand{\sing}{\operatorname{Sing}}

\title[Nonuniformly Sectional Expanding Systems]{Mostly Nonuniformly Sectional Expanding Systems}

\date{\today}

\author{Vitor Araujo and Luciana Salgado}

\thanks{V.A. and L.S. are partially supported by Projeto Universal
  CNPq-Brazil (grant 401737/2025-0).  V.A. was also partially
  supported by CNPq-Brazil (grant 304047/2023-6). L.S. also was
  partially supported by FAPERJ-Funda\c c\~ao Carlos Chagas Filho de
  Amparo \`a Pesquisa do Estado do Rio de Janeiro Project
  JCNE-E-26/200.271/2023 SEI-260003/000640/2023, by Coordena\c c\~ao
  de Aperfei\c coamento de Pessoal de N\'ivel Superior CAPES — Finance
  Code 001 and by PROEXT-PG project Dynamic Women - Din\^amicas,
  CNPq-Brazil (grant Projeto Universal 404943/2023-3).}

\address[V.A.]{Universidade Federal da Bahia,
Instituto de Matem\'atica\\
Av. Adhemar de Barros, S/N , Ondina,
40170-110 - Salvador-BA-Brazil}
\email{vitor.d.araujo@ufba.br \text{or} vitor.araujo.im.ufba@gmail.com}
\urladdr{https://sites.google.com/view/vitor-araujo-ime-ufba}

\address[L.S.]{Universidade Federal do Rio de Janeiro, Instituto de
   Matem\'atica\\
   Avenida Athos da Silveira Ramos 149 Cidade Universit\'aria, P.O. Box 68530, 
   21941-909 Rio de Janeiro-RJ-Brazil }
 \email{lsalgado@im.ufrj.br, lucianasalgado@ufrj.br}
 \urladdr{http://www.im.ufrj.br/~lsalgado}



\begin{document}

\maketitle

\begin{abstract}
  We introduce the notion of \emph{mostly nonuniform sectional
    expanding} (MNUSE) for singular flows which encompasses the
  notions of sectional hyperbolicity, asymptotically sectional and
  multisingular hyperbolicity. We construct examples of a vector field
  of class $C^r, r \ge 1$, whose flow exhibits a nonuniformly
  sectional hyperbolic set satisfying MNUSE, which is neither
  sectional hyperbolic nor asymptotically sectional hyperbolic.

  We obtain sufficient conditions for the existence of physical/SRB
  measures for asymptotically sectionally hyperbolic attracting sets
  with any finite codimension, extending the codimension two case.

  We provide examples of such attractors, either with non-sectional
  hyperbolic equilibria, or with sectional hyperbolic equilibria of
  mixed type, i.e., with a Lorenz-like singularity together with a
  Rovella-like singularity in a transitive set. These are
  higher-dimensional versions of contracting Lorenz-like attractors
  (also known as Rovella-like attractors) to which we apply our
  criteria to obtain a physical/SRB measure with full ergodic basin.

  We also adapt the previous examples to obtain higher codimensional
  (i.e. with central direction of dimension greater than $2$)
  nonuniformly sectional expanding attractors.
\end{abstract}

\tableofcontents

\section{Introduction and statement of results}
\label{sec:introduction}

A very important notion of (weak) hyperbolicity for flows has been
given by Morales, Pacifico and Pujals~\cite{MPP99} in the end of
1990's. The so called \emph{singular hyperbolicity} emerged from the
need of understanding the dynamical behavior of the famous Lorenz's
system and related ones. In this three-dimensional case, there is a
partially hyperbolic splitting of the tangent bundle into a pair of
continuous invariant subbundles $E\oplus F$, such that $E$ is
one-dimensional and uniformly contracting; and $F$, the central one,
is uniformly volume expanding (meaning that the Jacobian of the flow
along $F$ increases exponentially in time).  For a two-dimensional
$F$, this is the same as area expansion.

Labarca and Pacifico~\cite{LP86} introduced the singular horseshoe, a
variation of the geometric Lorenz attractor conceived to disprove
Palis-Smale's stability conjecture for flows on manifolds with
boundary. Later, Rovella~\cite{Ro93} introduced another variation of
geometric Lorenz attractor --- inspired by the family considered by
Lyubimov and Zaks~\cite{LybmvZks} and its bifurcations, recently studied
by Kazakov~\cite{Kazakov21} --- replacing the singularity by one with
a central contracting condition. These models are known as contracting
Lorenz models or simply Rovella attractors, and their singularities
known as ``Rovella-like''.

In order to extend these notions to higher dimensional systems,
Morales and Metzger~\cite{MeMor08} introduced the so called
\emph{($2$-)sectional hyperbolicity} (SH), which means uniform area
expansion along any $2$-plane in the central subbundle of a partially
hyperbolic splitting of a singular flow. The singular hyperbolic
attractors in $3$-manifolds, including the family of Lorenz
attractors, are naturally sectional hyperbolic. In higher dimensions,
the multidimensional Lorenz attractor introduced by Bonatti,
Pumari\~no and Viana~\cite{BPV97} is also sectional hyperbolic.

In higher dimensions, the notion of singular and ($2$-)sectional
hyperbolicity are quite different, since we can have volume expansion
without having area expansion; consider e.g. the ``wild strange
attractors'' for singular flows from Shilnikov and Turaev~\cite{ST98}
and also from Gonchenko, Kazarov and Turaev~\cite{Gonchenko_2021},
further studied by Pati\~no-Echeverr{\'{\i}}a, Krauskopf and
Osinga~\cite{PatioEcheverra2026}.

After this, Bonatti and da Luz in~\cite{BondaLuz21}, introduced the
notion of \emph{multisingular hyperbolicity} (MSH), attempting to
characterize the dynamical behavior of flows where singularities of
different indices coexist and, in particular, \emph{star flows}. These
are flows for which all non-trivial periodic trajectories for all
$C^1$ nearby flows are hyperbolic.  Bonatti and da Luz proved that
multisingular flows are star flows, and that there is a $C^1$ open
dense set of star flows which are multisingular. Recently, Crovisier
et al~\cite{Crovetal2020} provided a similar notion equivalent to
multisingular hyperbolicity under certain conditions.

We note that all concepts of hyperbolicity given before allow for
flows with (hyperbolic) singularities accumulated by regular orbits.

To better describe the dynamics of singular flows, a first notion of
nonuniform sectional hyperbolicity for flows has already appeared
in~\cite{ArbSal2011}, as an extension of a similar one for
diffeomorphisms by Castro in~\cite{Castro2011}, to the sectional
hyperbolic context.  The authors obtained the sectional hyperbolicity
of the nonwandering set $C^1$ generically among flows with nonuniform
sectional hyperbolic critical set.

However, this notion, as announced, is too strong, since it cannot
hold for the persistent hyperbolicity of the contracting Lorenz (also
known as Rovella-like) attractors~\cite{Ro93}. Morales and San Martin
introduced~\cite{MorSM17} the definition of asymptotically sectional
hyperbolicity (ASH) which is suitable to describe
this type of singular flows.

Here, we give a new definition of \emph{mostly nonuniformly sectional
  expanding systems} (MNUSE) which encompasses all the previous ones. 

For flows generated by $C^2$ vector fields we show that sectional
hyperbolic attracting sets, and asymptotically sectional hyperbolic
attracting sets with two-dimensional central bundle, satisfy MNUSE.
Moreover, these sets support a physical/SRB measure, which is unique
if the set is transitive.

We also describe the construction of examples of a vector fields of
class $C^r, r \ge 1$, whose flow exhibits a nonuniformly sectional
hyperbolic set satisfying MNUSE, which is neither sectional hyperbolic
nor asymptotically sectional hyperbolic.

We obtain sufficient conditions for the existence of
physical/SRB measures for asymptotic sectionally hyperbolic attracting
sets with any finite codimension, thus generalizing the known results
for codimension two, allowing us to handle attracting sets with slow
recurrence to equilibria and weak asymptotic sectional expansion.

We construct new examples of higher codimensional attractors that are
asymptotically sectionally hyperbolic and either contain
non-sectionally hyperbolic equilibria; or combine Lorenz-like and
Rovella-like singularities in the same transitive set; to which we
apply our existence result for physical/SRB measures with full ergodic
basins.

Moreover, we adapt these constructions to provide attractors with
non-uniformly sectionally expanding central direction in higher
codimension, that is, the central subbundle may have any dimension
greater than $2$.

In the next subsection, we introduce our main results whose definitions
are given in Section~\ref{sec:prelim-definit}.

\subsection{Statement of results}
\label{sec:statement-results}

Let $M$ be a $C^\infty$ compact connected manifold with dimension
$\dim M=n$, endowed with a Riemannian metric, induced distance $d$ and
volume form (Lebesgue measure) $\m$.  Let also $\fX^r(M)$, $r\ge1$, be
the set of $C^r$ vector fields on $M$ endowed with the $C^r$ topology
and denote by $X_t$ the flow generated by $X\in\fX^r(M)$. We assume
that $X$ is inwardly transverse to the boundary $\partial M$ if
non-empty.

We first show that several main classes of singular attracting sets,
whose definitions we present in Section~\ref{sec:prelim-definit}, are
MNUSE sets for smooth vector fields.

\begin{mainthm}\label{mthm:SH-ASH-NUSH}
  Let $\Lambda \subset M$ be a compact invariant subset for a $C^2$
  vector field. Then the following classes of attracting sets are
  mostly non-uniformly sectional hyperbolic (MNUSE) sets:
  \begin{enumerate}
  \item Sectional hyperbolic (SH) attracting sets;
  \item Asymptotically sectional hyperbolic (ASH) attracting sets whose
    singularities are hyperbolic of saddle-type with either
    \begin{enumerate}
    \item two-dimensional center-unstable bundle; or
    \item all singularities are singular hyperbolic, i.e., 
      $|\det(Df\mid_{E^{cu}_\sigma})|\ge1$ for each
      $\sigma\in\sing_\Lambda(G)$; or
    \item a positive volume subset of weak slow recurrent points.
    \end{enumerate}
  \item Multisingular hyperbolic (MSH) attracting sets whose
    singularities are all active and with the same index.
  \end{enumerate}
\end{mainthm}

\begin{remark}[ASH and MSH attractors]
  \label{rmk:attractors}
  ASH attractors (transitive attracting sets) admit no sink nor source
  singularities, thus all possible hyperbolic singular points are
  automatically of saddle-type.
  
  Each singularity $\sigma$ of a MSH attractor $\Lambda$ is hyperbolic
  and Lorenz-like by definition, thus its index is the dimension of
  the stable singular-dominated direction; and $\sigma$ is active
  since the dense regular trajectory in $\Lambda$ implies the
  existence of
  $x\in W^{s}(\sigma)\cap \big(\Lambda \setminus \{\sigma\}\big)$ and
  of $y\in W^{u}(\sigma)\cap \big(\Lambda \setminus \{\sigma\}\big)$;
  see Section~\ref{sec:prelim-definit} for more details.
\end{remark}

Next results show that the class of multisingular hyperbolic
systems, given in~\cite{Crovetal2020}, does not encompass all sectional
hyperbolic systems, and provides examples of MNUSE 
which do not satisfy any of the SH, ASH and MSH conditions.

\begin{mainthm}[Star property versus MSH and Lorenz-like
  singularities]
  \label{mthm:nonmsh}
  On every $3$-dimen\-sional Riemannian manifold $M$, for any
  given $r\ge1$, there exists a $C^r$ open set of star vector fields
  exhibiting a sectional hyperbolic (SH) attracting set which is not
  multisingular hyperbolic (MSH). 

  In addition, there exists a $3$-dimensional connected Riemannian
  manifold $M$ which, for any $r\ge1$, admits a $C^r$ open set of star
  vector fields supporting a sectional hyperbolic attracting set
  containing no Lorenz-like singularities.
\end{mainthm}

We construct several examples of attracting sets and attractors which
satisfy only one of the  SH, ASH or MSH conditions; or that
admits dense trajectories with coexisting Lorenz-like and Rovella-like
equilibria, with two-dimensional central-unstable direction, or higher
central dimension.

\begin{mainthm}[examples of $3$-dimensional attracting
  sets]\label{mthm:nonsec3d}
  On every $3$-dimensional connected Riemannian manifold $M$ there
  exists a vector field $X$ exhibiting a MNUSE attracting set, which
  \begin{enumerate} 
  \item $X$ is of class $C^2$, is ASH but is neither SH nor MSH;
    and also
  \item $X$ is of class $C^3$, is neither SH, nor ASH, nor MSH.
  \end{enumerate}
\end{mainthm}

\begin{remark}
  \label{rmk:noHypLema}
  The attracting sets constructed in item (2) of
  Theorem~\ref{mthm:nonsec3d} do not satisfy the Hyperbolic
  Lemma~\ref{le:hyplemma}; see next section.
\end{remark}

The following are new higher-dimensional examples.

\begin{mainthm}[examples of $d$-dimensional attracting sets,
  $d>3$]\label{mthm:nonsec}
  There are compact boundaryless manifolds $M$ such that, for any
  given $r\ge1$, we can find a $C^r$ vector field on $M$ admitting ASH
  attractors with
  \begin{enumerate}
  \item two-dimensional central bundle and non-sectional
    hyperbolic (``sectionally neutral'') equilibria (and $\dim M=4$);
  \item equilibria of mixed type: both Lorenz-like (sectionally
    expanding) and Rovella-like (sectionally contracting) equilibria
    (and $\dim M=4$);
  \item \emph{any given central dimension $d_{cu}>2$} and a
    pair of equilibria of either non-sectional hyperbolic type, or
    mixed type (and $\dim M\ge 5$).
  \end{enumerate}
  In addition, we construct examples of partially hyperbolic attractors
  with $\dim M\ge5$ which are MNUSE with
  \begin{enumerate}[resume]
  \item \emph{three-dimensional center bundle} $d_{cu}=3$ and
    hyperbolic equilibria which are non-sectional hyperbolic; and
  \item equilibria of mixed-type: generalized Lorenz-like and
    generalized Rovella-like.
  \end{enumerate}
  All the above examples are neither SH nor MSH, and admit a unique
  physical/SRB measure with full basin if the vector fields are at
  least of class $C^2$.
\end{mainthm}
 
In view of Theorem~\ref{mthm:SH-ASH-NUSH}, the class MNUSE contains
the SH, ASH and MSH classes. Therefore, all properties of attracting
MNUSE systems also hold for these classes of attracting sets.


Theorem~\ref{mthm:SH-ASH-NUSH} depends on the next result providing
sufficient conditions for the existence of physical measures for
partially hyperbolic attracting sets.

\begin{mainthm}[Physical measure for weak
  ASH attracting set]
    \label{mthm:hdASH}
    Let a $C^2$ vector field $G$ on $M$ and a trapping region $U$ be
    given containing a partially hyperbolic attracting set $\Lambda$
    so that every $x\in\Lambda$ not converging to any equilibrium
    satisfies non-uniform sectional expansion (NUSE).  We assume that
    $\Lambda$ contains only saddle-type hyperbolic equilibria.

    If either one of the next condition holds
    \begin{enumerate}[(A)]
    \item the central bundle is two-dimensional (codimension $2$ case);
    \item all singularities are singular hyperbolic:
      $|\det(Df\mid_{E^{cu}_\sigma})|\ge1$ for all
      $\sigma\in\sing_\Lambda(G)$; or
    \item there exists a positive volume subset $\Omega\subset U$ of
    weak slow recurrent (wSR) points;
    \end{enumerate}
  then there exists a physical/SRB  measure supported on
  $\Lambda$.

  If, in addition,  $\Lambda$ is transitive, then $\Lambda$ supports a
  unique physical/SRB probability measure whose basin covers a
  neighborhood of $\Lambda$.
\end{mainthm}

Since a physical/SRB measure $\mu$ supported on $\Lambda$ cannot have
atoms by definition, then we deduce the following.

\begin{mainclly}[existence of physical/SRB and weak slow recurrence]
  \label{mcor:wSR-SRB}
  In the same setting of Theorem~\ref{mthm:hdASH}, there exists a
  physical/SRB measure $\mu$ supported on $\Lambda$ if, and only if,
  weak slow recurrence holds on a positive volume subset
  $\Omega\subset B(\mu)$.
\end{mainclly}

\subsection{Comments and conjectures}
\label{sec:conjectures}
Although the authors, together with Arbieto~\cite{ArArbSal}, have
shown that a continuous invariant splitting $T_\Lambda M=E \oplus F$
with uniform contraction along $E$ and sectional-expansion along $F$,
must be a dominated splitting, next example shows that domination is
not a consequence of the other MNUSE conditions, even including the
existence of physical/SRB measures.

\begin{example}[Robustly transitive set not admitting a dominated
  splitting]
  \label{ex:nondomcont}
  We consider the suspension flow $\phi_t$, with constant roof of
  height $1$, of the diffeomorphisms $f:\TT^4\to\TT^4$ of the
  $4$-torus studied by Tahzibi in~\cite{tah04}. The diffeomorphism
  admits a $Df$-invariant dominated splitting
  $TM= E^{cs}\oplus E^{cu}$ together with a hyperbolic periodic point
  with expanding direction contained in $E^{cs}$, and another
  hyperbolic periodic point with contracting direction contained in
  $E^{cu}$. Hence, the naturally inherited $D\phi_t$-invariant
  splitting of the suspension
  $T\wt{M}=E^{cs}\oplus (E^{cu}\oplus(\RR\cdot X))$ is
  \emph{continuous but cannot be dominated}, where $\RR\cdot X$ is the
  flow direction on the special manifold $\wt{M}:= M/\sim$ defined by
  the natural dynamical identification.

  Moreover, since the time-$1$ map $f=X_1$ satisfies $P^1=Df$, then
  $E^{cu}=N^{cu}$ and, by the constructions presented in~\cite{tah04},
  we get~\eqref{eq:MNUE} Lebesgue almost everywhere. It also
  satisfies average asymptotic contraction along $E^{cs}$. In
  addition, since $f$ is transitive and admits a unique ergodic
  physical/SRB invariant measure, then the suspension of this measure
  is also physical/SRB with respect to the suspension flow $X_t$; see
  e.g.~\cite[Chpt. 7, Subsection 3.6]{AraPac2010}.
\end{example}

Using the same techniques from Mi, Cao and Yang~\cite{CMY2017}, we
may replace the domination condition by H\"older continuity of the
splitting over $\Lambda$, keeping the conclusion of
Theorem~\ref{mthm:hdASH}.

Recently, Cao, Mi and Zou~\cite{caomizou} propose a
method to construct Pesin unstable manifolds for hyperbolic measures
admitting a continuous splitting of the tangent bundle, so it should be
possible to replace the domination condition by only continuity of the
splitting, and still keep the conclusion of Theorem~\ref{mthm:hdASH}.

Items (A)-(B) of Theorem~\ref{mthm:hdASH}, restricted to $C^\infty$
diffeomorphisms, together with recent results from Burguet and
Ovadia~\cite{BOvadia}, show that the existence of physical/SRB
measures should only depend on positive Lyapunov exponents and not on
slow recurrence.

\begin{conj}[physical/SRB without slow recurrence]
  \label{conj:physnoSR}
 In Theorem~\ref{thm:discretefabv} and Theorem~\ref{mthm:hdASH} the
 slow recurrence conditions are superfluous.
\end{conj}

Extending the theory of ASH sets,
from~\cite{APPV,ArCerq,APRV24,RegoVivas24} and the results of the
present work, we propose the following.

\begin{conj}
  \label{conj:hdASH}
  ASH attracting sets with any codimension are entropy-expansive,
  kinematic-expansive, rescaling-expansive, and contain a dense set of
  periodic orbits under mild conditions.
\end{conj}

\subsection{Organization of the text}
\label{sec:organization-text}

The text is organized as follows, in Section \ref{sec:prelim-definit}
we present precise definitions of the SH, ASH and MSH classes,
together with auxiliary results. In
Section~\ref{sec:proof-main-theorems} we prove
Theorems~\ref{mthm:SH-ASH-NUSH},~\ref{mthm:nonmsh} and
Theorem~\ref{mthm:hdASH}; and also items (1)-(2) of
Theorem~\ref{mthm:nonsec}.  In Section~\ref{sec:multid-exampl} we
construct examples proving items (3)-(5) of
Theorem~~\ref{mthm:nonsec}; while in
Section~\ref{sec:new-examples-non} we construct examples proving items
(6)-(7).

\subsection*{Acknowledgements}
\label{sec:acknowledgements}

V.A.  thanks the Mathematics and Statistics Institute of the Federal
University of Bahia (Brazil) for its support of basic research. L.S.
thanks the Mathematics Institute of Universidade Federal do Rio de
Janeiro (Brazil) for its encouraging of mathematical research; and the
Mathematics and Statistics Institute of the Federal University of
Bahia (Brazil) for its hospitality. The authors thank the anonymous
referee for the careful reading of the text and valuable suggestions.

\section{Definitions and auxiliary results}
\label{sec:prelim-definit}

Let $X\in \fX^r(M)$ for some $r\ge1$ be given.
We say that $\sigma\in M$ is an {\em equilibrium} or
\emph{singularity} if $X(\sigma)=0$ and denote by $\sing(X)$ the
family of all such points. An equilibrium
$\sigma\in\sing(X)$ is \emph{hyperbolic} if all the eigenvalues of
$DX(\sigma)$ have non-zero real part.

An \emph{invariant set} $\Lambda\subset M$ for the flow $X_t$,
generated by the vector field $X$, is a subset satisfying
$X_t(\Lambda)=\Lambda$ for all $t\in\RR$. A point $p\in M$ is
\emph{periodic} if it is a \emph{regular point}, that is
$p\in M\setminus\sing(X)$, and there exists $\tau>0$ so that
$X_\tau(p)=p$; its orbit
$\SO_X(p)=X_{\RR}(p)=X_{[0,\tau]}(p)=\{X_t(p): t\in[0,\tau]\}$ is a
\emph{periodic orbit} --- an invariant simple closed curve for the
flow.  An invariant set is \emph{nontrivial} if it is not a finite
collection of periodic orbits and equilibria.  The family $\crit(X)$
of all equilibria and periodic orbits of a vector field $X$ is
\emph{the critical set} of $X$.

We say that a compact invariant set $\Lambda$ is \emph{isolated} if
there exists an open set $U\supset \Lambda$ so that
$ \Lambda =\bigcap_{t\in\RR}\close{X_t(U)}$; and $\Lambda$ is an
\emph{attracting set} if $U$ can be chosen so that
$\close{X_t(U)}\subset U$ for all $t>0$ --- in which case $U$ is a
\emph{trapping region} (or \emph{isolating neighborhood}) for
$\Lambda=\Lambda_X(U):=\cap_{t>0}\close{X_t(U)}$.

An \emph{attractor} is a \emph{transitive} attracting set, i.e., an
attracting set $\Lambda$ with a point $z\in\Lambda$ so that its
positive trajectory
$\SO^+(z)=X_{[0,+\infty)}(z)=\{X_t(z):t\ge0\}$ is dense:
$\closure{\SO^+(z)}=\Lambda$.

\subsection{Singular hyperbolic structures}
\label{sec:sing-hyp-strct}

We are interested in study (singular) hyperbolic structures, the
relation between non-uniform and uniform hyperbolic ones as well as
ergodic and dynamical consequences for flows.  In the following,
several notions of singular hyperbolic structures and some of their
features are presented.

\subsubsection{Partial hyperbolic attracting sets for vector fields}
\label{sec:part-hyperb-diff}

Let $\Lambda$ be a compact invariant set. We say that $\Lambda$ is
\emph{partially hyperbolic} if the tangent bundle over $\Lambda$
splits as a continuous $DX_t$-invariant Whitney sum
$T_\Lambda M=E^s\oplus E^{cu}$ such that we have domination of the
splitting and uniform contraction along $E^s$: that is, there exists a
constant $\lambda >0$ such that 
\begin{align*}
  \max\{\|DX_t | E^s_x\|,
  \|DX_t | E^s_x\| \cdot \|DX_{-t} | E^{cu}_{X_tx}\|\}
  \le
  e^{-\lambda t},
  \quad \forall x \in \Lambda, t\ge0;
\end{align*}
where $d_s:=\dim E^s_x\ge1$ and
$d_{cu}:=\dim E^{cu}_x\ge2$.

We refer to $E^s$ as the stable bundle and to $E^{cu}$ as the
center-unstable bundle.

\begin{remark}[adapted metric]
  \label{rmk:adaptmetric}
  We assume that the above holds for some choice of the Riemannian
  metric on the manifold. Changing the metric does not change the rate
  $\lambda$ but might introduce the multiplication by a constant; see
  e.g. \cite{Goum07}.
\end{remark}

In the vector field setting, a dominated splitting is automatically
partially hyperbolic whenever the flow direction is contained in the
central-unstable bundle $X\in E^{cu}$.  In fact, this inclusion is
equivalent to partial hyperbolicity; see e.g.~\cite[Lemma
3.2]{arsal2012a}.  Since the flow direction is invariant, partial
hyperbolicity is the natural setting to consider when studying
invariant sets (which are not composed only of equilibria) for flows
with a dominated splitting.

A {\em partially hyperbolic attracting set} is a partially hyperbolic
set that is also an attracting set.

\subsubsection{Singular/asymptotical sectional hyperbolicity and
  uniform hyperbolicity}
\label{sec:singul-hyperb-asympt}

We say that the center-unstable bundle is \emph{volume expanding} if
we can find constants $K,\lambda>0$ so that
$|\det(DX_t| E^{cu}_x)|\geq K e^{\lambda t}$ for all $x\in \Lambda$,
$t\geq 0$.

An invariant compact subset $\Lambda$ is \emph{(uniformly)
  hyperbolic} if it is partially hyperbolic and the central-unstable
bundle admits a continuous splitting $E^{cu}=(\RR\cdot X)\oplus E^u$,
with $\RR\cdot X$ the one-dimensional invariant flow direction and
$E^u$ a uniformly expanding subbundle. That is, we get the following
dominated splitting $T_\Lambda M= E^s\oplus(\RR\cdot X)\oplus E^u$ into
three subbundles; see e.g.~\cite{fisherHasselblatt12}.

A compact nontrivial partially hyperbolic invariant set $\Lambda$ with
volume expanding center-unstable bundle and whose equilibria are all
hyperbolic is a \emph{singular hyperbolic set}.  A singular hyperbolic
set which is also an attracting set is a \emph{singular
  hyperbolic attracting set}.

More generally, we say that $E^{cu}$ is \emph{($2$-)sectionally
  expanding} if there are positive constants $K , \lambda$ so that
for every $x \in \Lambda$ and each $2$-dimensional linear subspace
$L_x \subset E^{cu}_x$ one has $|\det(DX_t| L_x)|\geq K e^{\lambda t}$ for
all $t\geq 0$.

A \emph{sectional hyperbolic (attracting) set} (SH) is
a partially hyperbolic (attracting) set whose central subbundle is
sectionally expanding.

An \emph{asymptotically sectional hyperbolic (attracting) set } (ASH)
is a partially hyperbolic (attracting) set $\Lambda$ for which there
exists $\lambda>0$ so that
\begin{align}\label{eq:ASH}
  \limsup_{T\nearrow\infty}\frac1T
  \log|\det(DX_T\mid_{F_x})|\ge \lambda
\end{align}
for every
$x\in \Lambda^*:=\Lambda\setminus \cup\{W^s_\sigma:\sigma\in
U\setminus\sing(X)\}$ and each $2$-dimensional linear subspace $F_x$
of $E^{cu}_x$. That is, we have asymptotic sectional expansion along
all trajectories inside the invariant set which do not converge to any
singularity.

\subsection{Multisingular hyperbolicity}
\label{sec:multi-singul-hyperb}
  
A first definition of multisingular hyperbolicity was given in
\cite{BondaLuz21}.  Here, we are going to deal with multisingular
hyperbolicity (MSH) according to Crovisier et al~\cite{Crovetal2020}.
We need the following preliminary notion.

\subsubsection{Linear Poincar\'e Flow}
\label{sec:linear-poincare-flow}

Given  a regular point $x$ of the vector field $X$, we denote by
$ N_x=\{v\in T_xM: \langle v , X(x)\rangle=0\} $ the orthogonal
complement of $X(x)$ in $T_xM$; and  denote by $O_x:T_xM\to N_x$ the
orthogonal projection of $T_x M$ onto $N_x$.
We define for each $t\in\RR$
\begin{align*}
P_x^t:N_x\to N_{X_t x}
\quad\text{by}\quad
P_x^t=O_{X_t x}\circ DX_t(x);
\end{align*}
see Figure~\ref{fig:LPF}.

\begin{figure}[h]
  \centering
  \includegraphics[width=10cm]{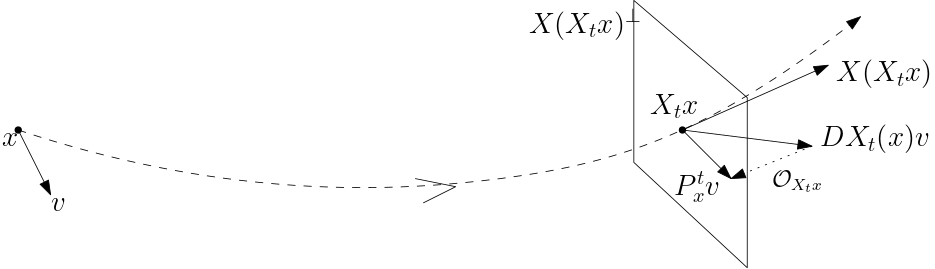}
  \caption{\label{fig:LPF}A representation of the Linear Poincar\'e
    flow $P^t_x$ of a vector $v\in T_xM$ with
    $x\in M\setminus\sing(X)$ and the orthogonal projection
    $O_{X_t x}:T_{X_tx}M\to N_{X_tx}$, with
    $N_{X_tx}=X(X_t x)^\perp$.}
\end{figure}

It is straightforward to check that $P=\{P_x^t:t\in\RR, X(x)\neq 0\}$
satisfies the relation $ P^{s+t}_x=P^t_{X_s x}\circ P^s_x $ for every
$t,s\in\RR$.  The family $P=P_X$ is called the {\em Linear Poincar\'e
  Flow} (LPF) of $X$.

\subsubsection{Singular-domination}
\label{sec:singular-domination}

We say that a LPF invariant decomposition of the normal bundle
$\mathcal{N}_{\Lambda \setminus \sing(X)}=\mathcal{N}_1 \oplus
\mathcal{N}_2$ is a \emph{singular dominated splitting} of index $i$
if:
\begin{enumerate}
\item $\mathcal{N}_1 \oplus \mathcal{N}_2$ is dominated of index $i$.
\item over each singularity $\sigma$ of $\Lambda \cap \sing(X)$ one of
  the following options hold:
\begin{itemize}
\item either there is a $DX_t$-dominated splitting
  $T_{\sigma}M = E^{ss}\oplus F$ for which $E^{ss}$ is uniformly
  contracting, $\dim(E^{ss})=\dim (\mathcal{N}_1)$ and
  $W^{ss}(\sigma)\cap \Lambda = \{\sigma\}$;
\item or there is a $DX_t$-dominated splitting
  $T_{\sigma}M = E\oplus E^{uu}$ for which $E^{uu}$ is uniformly
  expanding, $\dim(E^{uu})=\dim (\mathcal{N}_2)$ and
  $W^{uu}(\sigma)\cap \Lambda = \{\sigma\}$.
\end{itemize}
\end{enumerate}

Another important definition is the next.

\begin{definition}\label{def:Lorenz-sing}
  A singularity $\sigma \in \sing(X)$ is said to be:
  \begin{description}
  \item[Lorenz-like] if we have the invariant splitting
    $ T_\sigma M = E^{ss}\oplus E^{c} \oplus E^{uu} $ such that
    $\dim(E^c)=1$, then the smallest Lyapunov exponent $\lambda^{u}$
    of $E^{uu}$, the largest Lyapunov exponent $\lambda^{s}$ of
    $E^{ss}$, and the center Lyapunov exponent $\lambda^c$ satisfy
    $ 0 \leq \vert\lambda^c\vert < \min\{ -\lambda^{s},\lambda^{u}\}.
    $
 \item[Rovella-like]
   if, in the same setting as above, we have
   $|\lambda^c|>\max\{\lambda^u,-\lambda^s\}>0$.
 \item[Active] if $\sigma$ is hyperbolic and satisfies
   $W^{s}(\sigma)\cap \big(\Lambda \setminus \{\sigma\}\big)
   \neq
   \emptyset
   \neq  W^{u}(\sigma)\cap \big(\Lambda \setminus \{\sigma\}\big)$.
 \end{description}
\end{definition}

\begin{remark}[sectional expansion/contraction]
  \label{rmk:rovlike}
  If $\lambda^c<0$, then $\lambda^c+\lambda^u>0$ (sectional expansion
  of $Df$ at $E^{cu}_\sigma:=E^c_\sigma\oplus E^u_\sigma$) for
  Lorenz-like equilibria; and $\lambda^c+\lambda^u<0$ (sectional
  contraction of $Df$ at $E^{cu}_\sigma$) for Rovella-like equilibria.
\end{remark}

Now we are able to introduce the definition of multisingular
hyperbolicity following Crovisier et al~\cite{Crovetal2020}.

\begin{definition}\label{def:MSH}
  We say that an invariant compact subset $\Gamma \subset M$ is
  \emph{multisingular hyperbolic} (MSH) if
\begin{enumerate}
\item $\Gamma$ has a singular dominated splitting
  $\mathcal{N}^s \oplus \mathcal{N}^u$;
\item there exist constants $\lambda,T>0$ and an isolating
  neighborhood $U$ of $\Gamma \cap \sing(X)$ such that, for each
  $x\in\Gamma\setminus U$ with $X_t(x)\in\Gamma\setminus U$ for all
  $t>T$, we have
\begin{align*}
  \max\{ \Vert P^t\vert_{\mathcal{N}^s(x)}\Vert,
  \Vert P^{-t}\vert_{\mathcal{N}^u (X_t(x)) }\Vert\}
  \le e^{-\lambda t};
\end{align*}
\item and any singularity $\sigma$ in $\Gamma$ is Lorenz-like with a
  splitting
  $T_{\sigma}M = E^{ss}_\sigma \oplus E^c_\sigma \oplus E^{uu}_\sigma$
  with $\dim(E^{ss}_\sigma) = \dim (\mathcal{N}^s(x))$ and
  $\dim(E^{uu}_\sigma) = \dim (\mathcal{N}^u(x))$.
 \end{enumerate}
 \end{definition}

\subsubsection{Weak versus uniform hyperbolicity - Hyperbolic Lemma}\label{sec{hyp-lemma}}

We stress that all the above cited notions of singular hyperbolicity
involve an extension of the uniform ones. This is clear by the next
result.

\begin{lemma}[Hyperbolic Lemma]\label{le:hyplemma}
  Every compact invariant subset $\Gamma$ without equilibria contained
  in either a SH, or a ASH, or a MSH set is uniformly hyperbolic.
\end{lemma}
\begin{proof}
  See e.g.~\cite[Proposition 1.8]{MPP04} for SH sets;
  and~\cite[Theorem 2.2]{SmartinVivas20} for the ASH. For MSH
  non-singular set, cf. \cite[Theorem C]{Crovetal2020}.
\end{proof}

\subsection{Non-uniform sectional hyperbolicity}
\label{sec:asympt-singul-hyperb}

We recall the definition of (strong) non-uniform sectional
hyperbolicity introduced in~\cite{ArbSal2011}, by one of the authors
jointly with A. Arbieto, where we make a slight modification in the
splitting dropping the domination condition. Just after this, we
introduce a new definition of weaker hyperbolicity of singular vector
fields.

\begin{definition}[non-uniform sectional hyperbolicity]
  \label{def:wNUSH}
  Let the flow $X_t$ of a $C^1$ vector field $X$ be given and
  $f=X_\tau$ be the discrete time map for some given fixed $\tau>0$.
 A compact invariant subset $\Lambda \subset M$ is a
    \emph{non-uniformly sectional hyperbolic (NUSH)} set if
    \begin{enumerate}[(a)]
    \item there exists a continuous $DX_t$-invariant dominated
      splitting $T_{\Lambda}M= E \oplus F$;
    \item there exists $\eta < 0$ such that for any $p \in \Lambda$
      \begin{align}\label{eq:intNUC}
        \liminf_{t \to \infty} \frac{1}{t} \int_{0}^{t}
        \log\|Df|_{E_{X_s(p)}}\|\, ds
        &\leq \eta, \qand
        \\
        \liminf_{t \to \infty} \frac{1}{t} \int_{0}^{t}
        \log\| [\wedge^2(Df|_{F_{X_s(p)}})]^{-1}\|\, ds
        &\leq \eta.     \label{eq:intNUE}
      \end{align}
    \end{enumerate}
\end{definition}

\begin{remark}[Non-uniform sectional expansion]
  \label{rmk:NUSE}
  A compact partially hyperbolic invariant set $\Lambda$
  satisfying~\eqref{eq:intNUE} is \emph{non-uniformly sectional
    expanding} (NUSE).
\end{remark}

\begin{remark}[Consequence for singularities and periodic orbits]
  \label{rmk:conseqNUSH}
  On a NUSH set, every periodic orbit $p$ is sectional hyperbolic,
  i.e.~if the period of $p$ is $t(p)$, then the orbit
  $\SO(p):\{X_t(p):0\le t\le t(p)\}$ admits a continuous
  $DX_t$-invariant dominated splitting $T_{\SO(p)}M= E \oplus F$ such
  that for some constant $\eta<0$
    \begin{align*}
      \frac{1}{t(p)} \int_{0}^{t(p)} \log\Vert Df\vert_{E_{X_s(p)}}\Vert ds
      &\leq \eta, \qand
      \\
      \frac{1}{t(p)} \int_{0}^{t(p)}
      \log \Vert[\wedge^2(Df\vert_{F_{X_s(p)}})]^{-1}\Vert ds &\leq \eta.
    \end{align*}
    Analogously, every singularity $\sigma$ is sectional hyperbolic:
    there exists a $Df$-invariant dominated splitting
    $T_\sigma M= E\oplus F$ such that
    \begin{align*}
      \log\|Df|_{E_{\sigma}}\|\leq \eta
      \qand
      \log\|[\wedge^2(Df|_{F_{\sigma}})]^{-1}\|
      \leq\eta.
    \end{align*}
\end{remark}

In \cite{ArbSal2011}, under the assumption of NUSH condition,
L. Salgado proved the following generic (Baire) type result, jointly
with A. Arbieto.

\begin{theorem}
  \label{thrm:nush-generic}
  Let $\SU\in\Mundo$ be an open subset. Suppose that for any $X$ in
  some residual subset $\SS$ of $\SU$, the critical set $\crit(X)$ is
  NUSH. Then, there exists a residual subset $\SA$ of $\SU$ such that
  every $X\in\SA$ is such that the non-wandering set $\Omega(X)$ is SH.
\end{theorem}

As motivation for our new definition of non-uniformly sectional
hyperbolicity: the present authors, jointly with S. Sousa
in~\cite[Theorem 1.8]{ArSal25}, show that a under total probability
assumption, NUSH subset $\Lambda$ is a (uniformly) sectional
hyperbolic set.

Next, we state this result considering the NUSH property for any point
in $\Lambda$, and we give the idea of proof here for completeness.
\begin{theorem}\label{thrm:nuhs-sh}
  Given a $C^1$ vector field, a compact invariant set
  $\Lambda \subset M$ such that every point satisfies the NUSH
  property, is a SH set for $X$.
\end{theorem}

\begin{proof}
  The splitting $T_\sigma M=E_\sigma\oplus F_\sigma$ at each
  equilibrium $\sigma\in\sing_\Lambda(X)$ is dominated by the NUSH
  assumption.  From the proof of the first statement of~\cite[Theorem
  1.8]{ArSal25}, we have that $E=E^s$ in uniformly contracted, and
  $F=E^{cu}$ is sectionally expanding.

  From the previous properties, the result~\cite[Theorem A]{ArArbSal}
  ensures that $T_\Lambda=E\oplus F$ is also a dominated splitting,
  completing all the properties of a sectional hyperbolic set.
\end{proof}


\begin{corollary}\label{le:nush-hyp-le}
  Every NUSH vector field satisfies the Hyperbolic
  Lemma~\ref{le:hyplemma}.
\end{corollary}
In view of the strength of the last NUSH definition, we
consider the following new non-uniform hyperbolic conditions.

\begin{definition}[Non-uniform hyperbolic condition]
  \label{def:NUSHvar}
  Let $U\subset M$ be a forward invariant set of $M$ for the flow of a
  $C^1$ vector field $X$ so that the maximal invariant subset
  $\Lambda=\Lambda_X(U)$ is partially hyperbolic.
  We write $f=X_\tau$ for the discrete time map of the flow of the
  vector field $X$ for some given fixed $\tau>0$. Then we say that
  $\Lambda$ is \emph{mostly non-uniformly sectional expanding} (MNUSE)
  if there exists $\Gamma\subset U$ with $\m(\Gamma)>0$ and $\eta < 0$
  such that, for some continuous extension of the central-unstable
  subbundle to $U$ (denoted by the same symbol), we have
  \begin{align}\label{eq:MNUE}
    \limsup_{T\nearrow\infty}\frac1T
    \int_0^T \log\|[\wedge^2 (Df\mid_{E^{cu}_{X_s x}})]^{-1}\| \,
    ds
    \le \eta,
    \quad x\in\Gamma.
  \end{align}
\end{definition}

These definitions imply that all transverse directions to the vector
field along the center-unstable subbundle have positive Lyapunov
exponent; that is, if $v\in E^{cu}_x\setminus(\RR\cdot X)$, then
$\chi(x,v):=\limsup_{t\to+\infty}\log\|DX_t(x)v\|^{1/t}>0$; see
e.g.~\cite[Theorem 1.6]{ArSal25}.

\begin{remark}[independence of extension]\label{rmk:indext}
  The above notions do not depend on the particular continuous
  extension of $E^{cu}_\Lambda$ to $U$ chosen, due to the domination
  of the splitting; see e.g. \cite[Proposition 2.3]{ArSal25}.
\end{remark}

\subsubsection{Versions of non-uniform expansion and recurrence}
\label{sec:versions-non-uniform}

The following notions are used in the proof of
Theorem~\ref{mthm:hdASH}.  We say that a partially hyperbolic
attracting set $\Lambda$ is
\begin{description}
\item[weak non-uniform $2$-sectionally expanding (wNU2SE)] if
  there exists $\eta<0$ so that
  \begin{align}\label{eq:wNU2SE}
    \Omega=\left\{
    x\in U: \liminf_{n\nearrow\infty}\frac1n\sum\nolimits_{i=0}^{n-1}
    \log
    \big\|[\wedge^2 (Df\mid_{E^{cu}_{f^i x}})]^{-1}\big\|
    \le \eta
    \right\}
    \qand \m(\Omega)>0.
  \end{align}
\end{description}
This is enough to ensure existence a physical/SRB measure under a slow
recurrence condition, as stated in Theorem~\ref{mthm:hdASH}.

\begin{remark}[strong version]
  Replacing $\liminf$ by $\limsup$ in~\eqref{eq:wNU2SE} we obtain the
  \emph{non-uniform $2$-sectional expansion} (NU2SE)
  condition~\eqref{eq:ASH} which provides a different sufficient
  condition for good statistical properties; see
  Theorem~\ref{thm:discretefabv} in the next
  Subsection~\ref{sec:physicalsrb-measures}.
\end{remark}

We say that the vector field $G$ has
\begin{description}
\item[slow recurrence (SR)] if, on a positive Lebesgue measure subset
  $\Omega\subset U$, for every $\epsilon>0$, we can find $r>0$ so
  that
\begin{align}
  \label{eq:SR}
  \limsup_{n\nearrow\infty}\frac1n\sum\nolimits_{i=0}^{n-1}-\log d_r
  \big(f^i(x),\sing_\Lambda(G)\big) <\epsilon, \quad x\in\Omega;
\end{align}
\end{description}
where $d_r(x,S)$ denotes de $r$-{\em truncated distance} from $x\in M$
to a subset $S$, that is
\begin{align*}
  d_{r}(x,S)=\left\{
  \begin{array}{lll}
d(x,S) & \textrm{if $0<d(x,S)< r$;}
\\
1 & \textrm{if $d(x,S)\geq r$.}
\end{array} \right.
\end{align*}
However, for Theorem~\ref{mthm:hdASH} a weaker version is enough. We
say that $G$ has
  \begin{description}
  \item[weak slow recurrence (wSR)] if, on the positive Lebesgue
    measure subset $\Omega\subset U$, for every $\epsilon>0$, we can
    find $r>0$ so that
\begin{align}
  \label{eq:wSR}
  \limsup_{n\nearrow\infty}\frac1n\sum\nolimits_{i=0}^{n-1}
  \delta_{f^ix}(B_r(\sing_\Lambda(G)))<\epsilon, \quad x\in\Omega.
\end{align}
  \end{description}

    \begin{remark}[no atoms at equilibria]
    \label{rmk:noatom}
    If $x\in U$ satisfies~\eqref{eq:wSR}, then any $f$-invariant
    probability measure $\mu$ obtained as a weak$^*$ accumulation of the
    empirical measures $\Big(\frac1n\sum_{i=0}^{n-1}
    \delta_{f^ix}\Big)_{n\ge1}$ does not admit the elements of
    $\sing_\Lambda(G)$ as atoms: $\mu(\sing_\Lambda(G))=0$.
  \end{remark}

  \begin{remark}\label{rmk:NUSE-wNU2SE}
    \begin{enumerate}
    \item (NUSE implies wNU2SE) From~\cite[Theorem 1.8]{ArSal25}, a
      trajectory not converging to any singularity and satisfying
      NUSE, also satisfies wNU2SE.
    \item ((w)NU2SE implies (w)ASH] From~\cite[Theorem 1.6]{ArSal25} a
      trajectory not converging to any singularity and satisfying
      wNU2SE (respectively, NU2SE) , also satisfies wASH (resp., ASH).
    \end{enumerate}
  \end{remark}


\subsection{Hyperbolic physical/SRB measures}
\label{sec:physicalsrb-measures}

Hyperbolicity of an invariant probability measure means
\emph{non-uniform hyperbolicity}. More precisely, the forward Lyapunov
exponent function
$
\chi^+(x,v):=\limsup\nolimits_{t\nearrow\infty}\log\|D\phi_t(x)v\|^{1/t}
$, for $v\in T_x M\setminus\{\vec0\}$ admits only finitely many values
$\chi^+_1(x)>...>\chi^+_{p(x)}(x)$ on $TM\setminus \{0\}$ (the
\emph{forward Lyapunov Spectrum} at $x$) and generates a filtration
$ 0\subsetneq V_{p(x)}(x) \subsetneq \cdots \subsetneq V_{1}=T_xM$
after setting $V_i(x):=\{ v\in T_xM, \ \chi^+(x,v)\leq \chi^+_i(x)\}$.
Here $p(x), \chi^+_i(x)$ are $\phi_t$-invariant functions; the vector
subspaces $V_i(x)$, $i=1,...,p(x)$ are $D\phi_t$-invariant; and each
of them depend Borel measurably on $x$.

An invariant probability measure $\mu$ supported in $\Lambda$ is
\emph{hyperbolic}, if the tangent bundle over $\Lambda$ splits as
$T_z M = E^s_z\oplus (\RR\cdot X(z)) \oplus F_z$ of invariant
subspaces defined for $\mu$-a.e.  $z$, where
$\RR\cdot X(z)$ is the flow direction (with zero Lyapunov exponent);
$\RR\cdot X(z)\oplus F_z=E^{cu}_z$ and $F_z$ is the direction with
positive Lyapunov exponents, i.e.
$\lim_{t\to+\infty}\frac1t\log\big\|(D\phi_t\mid_{F_z})^{-1}\big\|<0$.
All directions along $E^s_z$ have negative Lyapunov
exponents
$\lim_{t\to+\infty}\frac1t\log\big\|D\phi_t\mid_{E^s_z})\big\|<0$. This
ensures that the Lyapunov Spectrum satisfies
$\chi^+_{d_{cu}+1}(z)<0 = \chi^+_{d_{cu}}(z)<\chi^+_{d_{cu}-1}(z)$.

A \emph{physical measure} is an invariant probability measure $\mu$
for which asymptotic time averages exist and coincide with the space
average, for a set of initial conditions with positive Lebesgue
measure, i.e. in the weak$^*$ topology of convergence of probability
measures 
\begin{align*}
B(\mu):=\left\{z\in M: \lim_{T\nearrow\infty} \frac1T \int_0^T
  \delta_{\phi_t (z)} \,dt = \mu \right\};
\end{align*}
and the (ergodic) \emph{basin} $B(\mu)$ of the measure $\mu$ satisfies
$\Leb(B(\mu))>0$.

\begin{remark}[Existence of physical measures and MNUSE definition]
  \label{rmk:goodef}
  For $C^2$ vector fields, it is known that sectional hyperbolic
  attracting sets admit physical
  measures~\cite{APPV,LeplYa17,araujo_2021}, and so there exists a
  positive Lebesgue measure subset of their basin of attraction which
  satisfies the MNUSE property.
\emph{This shows that MNUSE is an adequate extension of uniform
  hyperbolicity for singular-flows.}
\end{remark}

A $cu$-\emph{Gibbs state} is an invariant probability measure which
satisfies
\begin{align}
  \label{eq:fPesin}
  h_\mu(X_1)
  =
  \int \Sigma^+\,d\mu
  =
  \int\log|\det DX_1\mid_{E^{cu}}|\,d\mu>0;
\end{align}
where $\Sigma^+(z)=\sum_{i\le\dim(E^u)} \chi_i^+(z)$ is the
sum of positive Lyapunov exponents (with multiplicities).

We will focus on \emph{hyperbolic physical/SRB} measures, that is,
invariant probability measures which are simultaneously hyperbolic
physical measures and $cu$-Gibbs states; see e.g~\cite{LY85}.

We have the following existence result for this class of invariant
probability measures.

\begin{theorem}[Physical/SRB measures for non-uniformly
  sectionally expanding flows]{\cite[Theorem B]{ArSal25}}
  \label{thm:discretefabv}
  Let $G\in\fX^2(M)$ be a vector field with a partially hyperbolic
  attracting set $\Lambda=\Lambda_G(U)$ satisfying (SR) on
  $\Omega\subset U$, with $\m(\Omega)>0$.  Then we have NU2SE on
  $\Omega$ if, and only if, there are finitely many ergodic
  physical/SRB measures whose basins cover $\m$-a.e. point of
  $\Omega$:
  $\m\Big(\Omega\setminus \big(B(\mu_1)\cup\dots\cup
  B(\mu_p)\big)\Big)=0$.
\end{theorem}





\section{Proof of the main statements}
\label{sec:proof-main-theorems}

Here we present first a proof of Theorem~\ref{mthm:hdASH} and, using
its statement, a proof of Theorem~\ref{mthm:SH-ASH-NUSH}. Then we
exhibit examples to prove Theorem~\ref{mthm:nonmsh} and the first part
(items (1)-(2)) of Theorem~\ref{mthm:nonsec}.

\subsection{Proof of Theorem~\ref{mthm:hdASH}}
\label{sec:proof-thmD}

We have a partially hyperbolic attracting subset whose equilibria, if
any, are hyperbolic of saddle type, for a $C^2$ vector field.

From Remark~\ref{rmk:NUSE-wNU2SE} the trajectories satisfying NUSE
also satisfy wNU2SE and so, if we have condition (A): the central
bundle is two-dimensional ($d_cu=\dim E^{cu}=2$ then, from
\cite[Corollary G]{ArSal25}, there exists a physical/SRB measure, as
stated.

For conditions (B)-(C), the proof relies on applying the following
useful extension of Pesin's Formula~\cite{Pe77} obtained by
Catsigeras-Cerminara-Enrich~\cite{CatCerEnr15}.

\begin{theorem}[Generalized Pesin's Inequality~\cite{CatCerEnr15}]
  \label{thm:GenPesin}
  For any $C^1$ diffeomorphism $f$, if $\Lambda$ is an invariant
  compact set with a dominated splitting $T_\Lambda M= E\oplus F$,
  then for Lebesgue almost every point $x$ satisfying
  $\omega(x)\subset \Lambda$, the entropy of any weak$^*$ limit
  measure $\mu$ of the sequence
  $\big(\frac1n \sum_{i=0}^{n-1} \delta_{f^i(x)}\big)_{n\ge1}$ is
  bounded from below:
  \begin{align}\label{eq:gpesin}
    h_\mu(f)\geq \int \log|\det (Df\mid_{F})|\,d\mu.
  \end{align}
\end{theorem}

We used this in~\cite{ArSal25}, for $x$ satisfying the (wNU2SE)
condition, to find an ergodic hyperbolic physical/SRB measure as an
ergodic component of a limit measure $\mu$ as above. It is well known
from the work of Ledrappier-Young~\cite{LY85} that for $C^2$ systems
such measures are physical/SRB measures.

In what follows we write $J^{cu}f(w):=|\det(Df\mid_{E^{cu}_w})|$ and
$\psi^{cu}(w):= \log\big\|[\wedge^2 (Df\mid_{E^{cu}_w})]^{-1}\big\|$.

  We start by using Theorem~\ref{thm:GenPesin} to fix a full
  $\leb$-measure subset $X\subset U$ and for $x\in X$ we consider a
  weak$^*$ limit point $\mu$ of the sequence considered in Theorem
  ~\ref{thm:GenPesin}. Then we have
  $h_\mu(f)\ge\int\log J^{cu}f\,d\mu$ from~\eqref{eq:gpesin}.

  We want to obtain the reverse inequality to conclude Pesin's Formula
  $h_\mu=\mu(\log J^{cu}f)>0$ and invoke Ledrappier-Young's main
  result from~\cite{LY85} ensuring, since $f$ is a $C^2$
  diffeomorphism, that $\mu$ admits an ergodic component $\nu$ which
  is a physical/SRB measure, as needed.

    We consider the following two cases: either $\mu$ admits some atom ---
  necessarily a periodic point of $f$ --- or $\mu$ is non-atomic.

  In the latter case, then $\mu(\sing_\Lambda(G))=0$ and a
  $\mu$-generic point $z\in\Lambda$ cannot converge to any
  singularity. Indeed, otherwise we would have for any continuous
  observable $\vfi:M\to\RR$
  \begin{align*}
    \mu(\vfi)=\int\vfi\,d\mu
    =
    \lim_{n\to+\infty}\frac1n\sum\nolimits_{i=0}^{n-1}\vfi(f^iz)
  \end{align*}
  since $z$ is Birkhoff-generic for $\mu$, and
  $\mu(\vfi)=\vfi(\sigma)$ if $z\in W^s(\sigma)$ for some singularity
  $\sigma\in\sing_\Lambda(G)$. Because this holds for any continuous
  observable, we conclude $\mu=\delta_\sigma$ contradicting the
  assumption that $\mu$ is non-atomic.

  Therefore, such Birkhoff-generic $z$ must be in
  $\Lambda^*:=\Lambda\setminus
  \cup\{W^s_\sigma:\sigma\in\sing_\Lambda(G)\}$ and so, by assumption,
  such $z$ satisfies wNU2SE. In particular, using $\psi^{cu}$ as
  observable together with Birkhoff Ergodic Theorem, we obtain
  $\mu(\psi^{cu})\le\eta$.

  This ensures, in particular, that all Lyapunov exponents along
  $E^{cu}$ are either zero (along the flow direction $G$) or strictly
  positive. Hence, we get that
  $\mu(\log J^{cu}f)=\int\Sigma^+\,d\mu\ge-\eta$ gives the averaged sum
  of central-unstable Lyapunov exponents, and we can apply Ruelle's
  Inequality
  $
    h_\mu(f)\le\int\Sigma^+\,d\mu = \mu(\log J^{cu}f).
  $
  We conclude that $\mu$ satisfies Pesin's Formula as needed, and the
  existence of an ergodic component $\nu$ of $\mu$ which is a
  physical/SRB measure follows.

  We are left with the case where $\mu$ admits an atomic
  component. Let $\mu=\alpha\cdot\nu+\beta\cdot\xi$ with
  $\alpha+\beta=1, \alpha\ge0, \beta>0$ be the decomposition of $\mu$
  into a non-atomic component $\nu$ and a purely atomic component
  $\xi$.  We want to use our conditions (B) or (C) to show that this
  case cannot occur.

  Atomic components of invariant probability measures are at most
  denumerably many; such components for a continuous transformation
  are supported on periodic orbits.  Periodic orbits of the time-$1$
  map $f$ of the flow of $G$ are either equilibria $\sigma$ (fixed
  points of the flow) or periodic orbits $p$ of the flow with integer
  minimal period.

  If $p\in\supp\xi$ has minimal period $\ell$, then
  $\pi_p:=\frac1\ell\sum_{i=0}^{\ell-1}\delta_{f^ip}$ is an ergodic
  component of $\xi$ and
  $\pi_p(\log J^{cu}f)=\frac1{\ell}\log J^{cu}f^\ell(p)>0$
  since $p\in\Lambda^*$.

  We now assume condition (B). Then for any
  $\sigma\in\sing_\Lambda(G)\cap \supp\xi$ we have
  $\delta_\sigma(\log J^{cu}f)>0$ and we cannot have $\beta=1$, i.e.,
  $\mu$ cannot be purely atomic. Indeed, in the purely atomic case, we
  have $h_\mu(f)=0\ge\mu(\log J^{cu}f)=\xi(\log J^{cu}f)$ and we can
  write
  \begin{align}\label{eq:xilogJ}
    \xi(\log J^{cu}f)
    =
    \left(\sum\nolimits_{\sigma}t_\sigma\delta_\sigma
    +
    \sum\nolimits_{i\ge1}t_i\pi_{p_i}
    \right)(\log J^{cu})>0, 
  \end{align}
  where $\sum_{i\ge1}t_i+\sum_\sigma t_\sigma =1$ and
  $t_i,t_\sigma\ge0$. This contradicts the previous inequality.

  This contradiction ensures that there exists a non-atomic component
  with positive mass: that is, $\alpha>0$. Since $\nu=\mu-\xi$ is
  also $f$-invariant, we can write
  \begin{align*}
    h_{\mu}(f)
    &=
    \alpha\cdot h_{\nu}(f)+\beta\cdot h_{\xi}(f)
    =
    \alpha \cdot h_{\nu}(f)
    \\
    &\ge
    \mu(\log J^{cu}f)
    =
    \alpha\cdot \nu(\log J^{cu}f)+\beta\cdot \xi(\log J^{cu}f)
    \ge
    \alpha \cdot \nu(\log J^{cu}f).
  \end{align*}
  We have obtained an $f$-invariant non-atomic probability measure $\nu$
  satisfying $h_{\nu}(f)\ge \nu(\log J^{cu}f)$. We can now proceed
  with the same argument as before to obtain a physical/SRB measure.

  We now replace condition (B) with condition (C). Then equilibria
  cannot be atoms of $\mu$ by Remark~\ref{rmk:noatom} However,
  from~\eqref{eq:xilogJ} other periodic points $p$ are also excluded
  since $p\in\Lambda^*$. Therefore, we conclude that $\beta=0$ in this
  case and we recover that $\mu$ is non-atomic. The rest of the
  argument follows and the proof of existence of a physical/SRB
  measure is complete.

  Finally, if $\Lambda$ is transitive, then the ergodic basin
  $B(\mu)$ of $\mu$ covers a neighborhood of $\Lambda$, except perhaps a
  zero volume subset, as a consequence of the properties of $cu$-Gibbs
  states; see e.g.~\cite[Theorem 5.1]{ArSal25}.
This completes the proof of Theorem~\ref{mthm:hdASH}.
  
\subsection{Proof of Theorem~\ref{mthm:SH-ASH-NUSH}}
\label{sec:proof-theorem-refmth}

We first note that we can reduce item (3) to item (1) in the statement
of Theorem~\ref{mthm:SH-ASH-NUSH}, since from~\cite[Theorem
C]{Crovetal2020} all MSH sets whose singularities are active and with
the same index, must be sectional hyperbolic sets.

To prove item (1), we recall that since $\Lambda$ is a
sectional hyperbolic attracting set, then
$\|DX_t\mid_{E^s_x}\|\le e^{-\lambda t}$ for all $t>0$; and there
exists $T>0$ so that
$\|[\wedge^2(DX_T\mid E^{cu}_x)]^{-1} \|\le e^{\lambda T}/K<1$ for
all $x\in\Lambda$.

Moreover, sectional hyperbolic attracting sets for $C^2$ vector fields
support finitely many physical/SRB measures $\mu_1,\dots,\mu_p$ whose
basins cover $\m$-a.e. point of an attracting neighborhood $U$ of
$\Lambda$, that is,
$\m( U \setminus B(\mu_1)\cup\dots\cup B(\mu_p))=0$; see
e.g.~\cite{APPV,AraPac2010,LeplYa17,araujo_2021}.

Hence, for $\m$-a.e. $x$ in $U$ there exists an ergodic invariant
measure $\mu=\mu_i$ so that
\begin{align}
  \lim_{n\nearrow\infty}\frac1t\int_0^t
  \log\|DX_T\mid_{E^{cs}_{X_sx}}\|\,ds
  &=
  \int_{x\in\Lambda} \log\|DX_T\mid_{E^{cs}_{x}}\| \,d\mu
  \le
  -\lambda T
    \qand \label{eq:MNUC1}
  \\
  \lim_{t\nearrow\infty}\frac1t\int_0^t
  \log\|[\wedge^2(DX_T\mid_{E^{cu}_{X_sx}})]^{-1}\|\,ds
  &=
  \int_{x\in\Lambda} \log\|[\wedge^2(DX_T\mid_{E^{cu}_{x}})]^{-1}\|\,d\mu
  \le \log\frac{e^{-\lambda T}}K. \nonumber
\end{align}
Thus setting $f=X_T$ and
$\eta=\sup\{-\lambda T, \log(e^{-\lambda T}/K)\}$ and
$\Gamma=B(\mu_1)\cup\dots\cup B(\mu_p)$, we see that $\Lambda$ becomes
a MNUSE attracting set, since we can extend the stable bundle
$E^s_\Lambda$ and center-stable bundle $E^{cu}_\Lambda$ to $U$; see
e.g.~\cite{ArMel17,ArSal25}. This completes the proof of items (1)
and (3).

For item (2a), we use that ASH attracting sets $\Lambda$ whose
equilibria are all hyperbolic of saddle-type and whose center-unstable
directions are two-dimensional, support a physical/SRB measure $\mu$
whose basin $B(\mu)$ has positive volume inside the attracting
neighborhood $U$ of $\Lambda$, that is $\m( U \cap B(\mu))=0$;
see~\cite[Corollary G \& Theorem 1.8]{ArSal25}.

Since the subset $S:=\{W^s_x: \sigma\in\sing_\Lambda(X)\}$, the union
of the stable manifolds of the finitely hyperbolic equilibria of $X$
in $\Lambda$, all of saddle-type, is the union of finitely many
immersed smooth submanifolds, then its volume is zero. In particular
we have that $\mu(S)=0$, for otherwise $\mu$ would contain some
singular point $\sigma\in \sing_\Lambda(X)$ as an invariant atom,
which is impossible for an ergodic physical/SRB measure.

Indeed, arguing by contradiction, either $0<\mu_i(\sigma)<1$ and $\mu$
is not ergodic; or $\mu_i=\delta_\sigma$ and~\eqref{eq:fPesin} cannot
hold, since the entropy $h_{\mu_i}(X_1)$ is zero, but $E^{cu}_\sigma$
contains an expanding direction and so a positive Lyapunov exponent.

Then $\m$-a.e. $x\in B(\mu)$ does not belong to $S$ and we
obtain~\eqref{eq:MNUC1} analogously to the item (1) for any given
positive $T$.
Since we are assuming $\dim E^{cu}_\Lambda=2$, then we obtain
$\log\|[\wedge^2(DX_T\mid_{E^{cu}_x})]^{-1}\|=-\log|\det(DX_T\mid_{E^{cu}_x})|$
and the map $t\mapsto \log|\det(DX_t\mid_{E^{cu}_x})|$ becomes
\emph{additive}. By Kingman's Subadditive Ergodic Theorem
and~\ref{eq:ASH} 
\begin{align*}
  -\lambda
  \ge
  \int_{x\in\Lambda} \lim_{t\nearrow\infty}
  \frac1t\log|\det(DX_t\mid_{E^{cu}_x})|^{-1}\, d\mu
  =
  \inf_{t>0}\int_{x\in\Lambda}
  \frac1t\log|\det(DX_t\mid_{E^{cu}_x})|^{-1}\,d\mu
\end{align*}
and so there exists $t(\mu)>0$ so that for $T\ge t(\mu)$
\begin{align}\label{eq:tmu}
  \int_{x\in\Lambda} -\log|\det(DX_T\mid_{E^{cu}_x})|\,d\mu
  <-\lambda T.
\end{align}
Thus, for $x\in B(\mu)$ and since $\mu$ is ergodic for the flow
\begin{align*}
\lim_{t\nearrow\infty}\frac1t\int_0^t
  &\log\|\wedge^2(DX_T\mid_{E^{cs}_{X_sx}})^{-1}\|\,ds
  =
    -\lim_{t\nearrow\infty}\frac1t\int_0^t
    \log|\det(DX_T\mid_{E^{cu}_{X_s x}})|\,ds
  \\
  &=
    -\int_{x\in\Lambda} \log|\det(DX_T\mid_{E^{cu}_x})|\,d\mu
    < -\lambda T.
\end{align*}
Hence, we obtain MNUSE with $\Gamma=B(\mu)$,
$\eta=-\lambda T$ and $f=X_T$ as before.

Finally, condition (2b) corresponds to item (B) of
Theorem~\ref{mthm:hdASH}; and (2c) to item (C) of the same Theorem,
thus we obtain a physical/SRB measure supported on the attracting set
and can take $\Gamma=B(\mu)$ to obtain MNUSE.
This completes the proof of Theorem~\ref{mthm:SH-ASH-NUSH}.

\subsection{Proof of Theorem~\ref{mthm:nonmsh}}
\label{sec:examples}

The following is an example of a sectional hyperbolic attracting set
containing both Lorenz-like and non-Lorenz-like singularities.

\begin{example}[Connected Lorenz attracting set containing
  non-Lorenz-like equilibria]
  \label{ex:gLorenzNeq}
  Consider the maximal invariant subset $A$ inside the trapping
  ellipsoid $E$ of the Lorenz system of equations, which are sketched
  in Figures~\ref{fig:ellipsoid} and~\ref{fig:Lorenz3sing}.

  \begin{figure}[htpb]
    \centering
    \includegraphics[width=4.5cm]{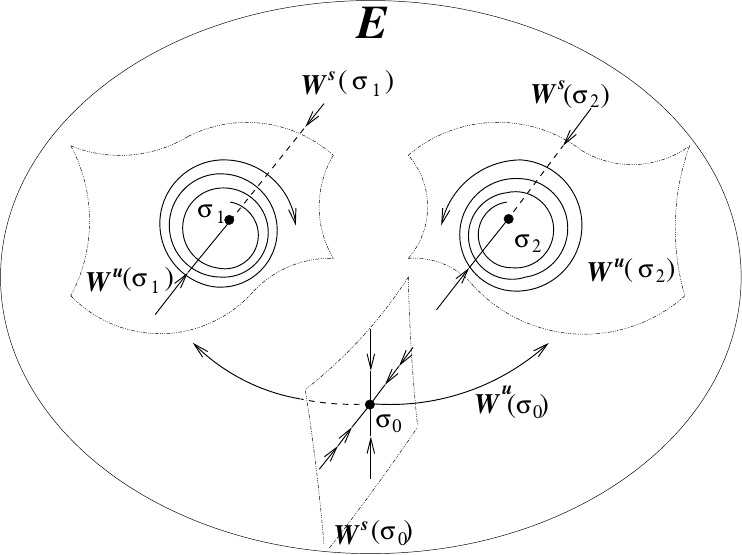}
    \qquad
    \includegraphics[width=4.5cm]{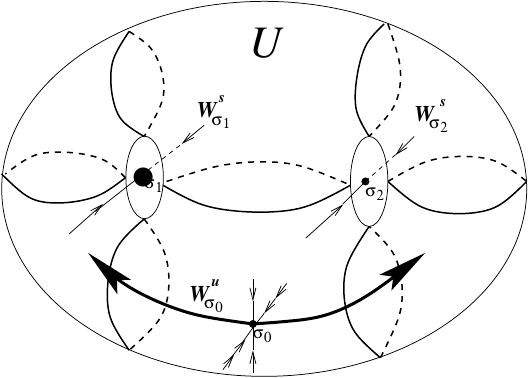}
    \caption{\label{fig:ellipsoid} Local stable and unstable
      manifolds near $\sigma_0, \sigma_1$ and $\sigma_2$, and the
      ellipsoid $E$ on the left hand side; the trapping bi-torus $U$ on
      the right hand side.}
  \end{figure}

  This set contains the classical Lorenz attractor $\Lambda$, which is
  the maximal invariant subset inside the trapping bi-torus $U$ of
  Figure~\ref{fig:ellipsoid} and sectional hyperbolic, together with
  the pair of hyperbolic saddle-type singularities away from the
  origin and inside the ``lobes'' of the ``butterfly''; see
  e.g.~\cite{AraPac2010}. These equilibria, denoted
  $\sigma_1,\sigma_2$ in Figures~\ref{fig:ellipsoid}
  and~\ref{fig:Lorenz3sing}, have a complex eigenvalue with positive
  real part. Hence, the two-dimensional unstable manifolds
  $W^u(\sigma_1)$ and $W^u(\sigma_2)$ are contained in the attracting
  set.

  \begin{figure}[htpb]
    \centering \includegraphics[width=5cm]{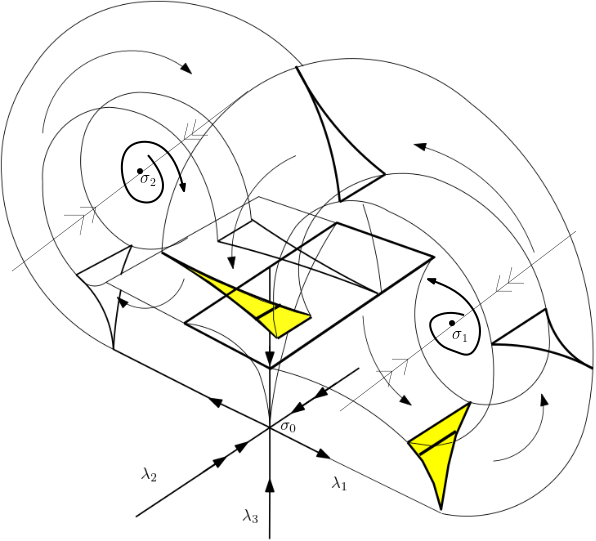}
    \caption{\label{fig:Lorenz3sing} The Lorenz attracting set
      including the geometric Lorenz attractor and the pair of
      hyperbolic saddle-type non-Lorenz like singularities
      $\sigma_1,\sigma_2$ with complex expanding eigenvalue.}
  \end{figure}

  This attracting set $A=\Lambda\cup W^u(\sigma_1)\cup W^2(\sigma_2)$
  is connected, since these unstable manifold accumulate inside the
  geometric Lorenz attractor, and clearly sectional hyperbolic (and
  singular hyperbolic) since the hyperbolic splittings of
  $T_{\sigma_1}\RR^3$ and $T_{\sigma_2}\RR^3$ are sectional hyperbolic,
  and  their unstable manifolds accumulate inside a
  sectional hyperbolic set.

  However, the equilibria $\sigma_1,\sigma_2$ are not Lorenz-like,
  either for the given vector field, or for its time-reverse.
\end{example}

From Example~\ref{ex:gLorenzNeq}, which is a vector field $Y$ of class
$C^\infty$, we can build a star vector field $X$ in any compact
$3$-manifold $M$ whose non-wandering set contains the geometric Lorenz
attractor, as follows. We start with a Morse-Smale $C^\infty$ gradient
flow $G$ with at least a hyperbolic source equilibrium $\sigma_0$ and
a hyperbolic sink equilibrium $\sigma$. Then we choose a small open
neighborhood $V$ of $\sigma$ and attach a rescaled version of the
vector field from Example~\ref{ex:gLorenzNeq} on the trapping
ellipsoid $E$.

More precisely, let $\psi: U\subset\RR^3\to M$ be a $C^\infty$ smooth
chart so that $U$ is an open neighborhood of $0$, $\psi(0)=\sigma$ and
$\psi(U)=V$, so that $S=\partial V$ is diffeomorphic to the $2$-sphere
$\sS^2$ and the vector field $G$ is inward transverse to $S$. Let
$\vfi:M\to[0,1]$ be a $C^\infty$ bump function such that
$\vfi^{-1}(\{1\})\subset V$ is a compact neighborhood of $0$ and
$\vfi^{-1}(\{0\}\supset \ov{M\setminus V}$.  Let $\rho\in(0,1)$ so
that $\rho E\subset U$ and define
\begin{align*}
  X(x) = (1-\vfi(x))\cdot G(x)+\vfi(x)\cdot D\psi(x)\cdot\Big(\rho
  Y\big(\rho^{-1}\psi^{-1}(x)\big)\Big).
\end{align*}
Since both $G$ and $Y$ are inwardly transverse at $S$ and
$\partial E$, then the critical elements of $X$ are the critical
elements of $G$, except $\sigma$, together with the
critical elements of $Y$ rescaled to the neighborhood $V$.

Then $X$ is a star flow. Moreover, the non-wandering set $\Omega(X)$
of $X$ contains the (rescaled) attracting set $A$ of
Example~\ref{ex:gLorenzNeq}, together with a hyperbolic source
singularity $\sigma_0$. Therefore, the compact invariant set $M$ for
$X$ is not MSH.

In addition, we note that
\begin{itemize}
\item all critical elements (singularities and periodic orbits) of
  this example are hyperbolic with stable and unstable manifolds
  transverse to the ellipsoid $E$ and bitorus $U$; and
\item in the interior of $U$ the maximal invariant subset is formed by
  a singular hyperbolic attractor, which is known to persist under
  $C^r$ perturbations of the given vector field, for any $r\ge1$, from
  the theory of singular hyperbolic flows~\cite{AraPac2010s};
\end{itemize}
thus the constructed example is $C^r$ robust, $r\ge1$, as claimed in
Theorem~\ref{mthm:nonmsh}.

Next is an example of sectional hyperbolic attracting set with
no Lorenz-like singularities.

\begin{example}[sectional hyperbolic with no Lorenz-like
  singularities]
  \label{ex:SHnoLsing}
  As in the construction of the anomalous (intransitive) Anosov flow,
  Morales~\cite{Morales07} proposed the following.
  \begin{description}
  \item[Step $0$]
  Suspend a Derived from Anosov diffeomorphism on the torus
  $\mathbb{T}^2$, getting an Axiom A vector field $X^0$ on a closed
  manifold $M^0$, whose nonwandering set is the union of a repeller
  periodic orbit $O$ and a hyperbolic attractor $A$.
\item[Step $1$] Choose a solid torus neighborhood $U$ of $O$ such that
  the stable foliation of $A$ intersects the boundary $\partial U$ in
  two Reeb components. The core $C$ of the top component is transverse
  to the intersection of the stable manifold of $A$ with $\partial U$;
  see Figure~\ref{fig:cnmsing}. Now, remove the interior of $U$ from
  $M^0$ obtaining a new compact manifold $M^1$ with boundary equal to
  $\partial U$. The restriction $X^1 = X^0\vert_{M^1}$ is transverse
  to $\partial M^1 = \partial U$ inwardly pointing to $M^1$. We note
  that the maximal invariant subset $\cap_{t>0}X^1_t(M^1)$ of $X^1$ is
  the hyperbolic attractor $A$ (where $(X^1_t)_{t\ge0}$ is the
  positive flow of $X^1$).
  \begin{figure}[htpb]
  \centering \includegraphics[width=5cm]{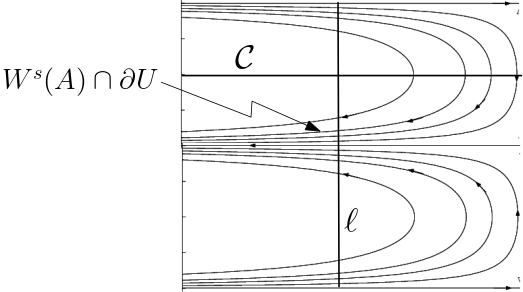}
   \qquad
  \includegraphics[width=5cm]{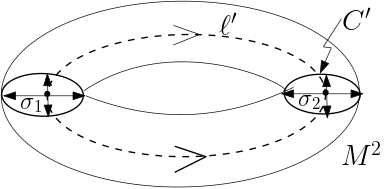}
  \caption{\label{fig:cnmsing} Construction of a sectional-hyperbolic
    with no Lorenz-like singularity.}
\end{figure}
\item[Step $2$] Consider a vector field $X^2$ on a solid torus $M^2$
presenting two hyperbolic singularities $\sigma_1$ (a repelling one)
and $\sigma_2$ (a saddle with one-dimensional stable direction). Note
that $X^2$ is transverse to the boundary $\partial M^2$ pointing
outwardly. Moreover, there is a meridian curve $C^\prime$ in
$\partial M^2$ which is the boundary of a disk contained in the
unstable manifold of $\sigma_2$.
\item[Step $4$]
  Finally, make a Dehn surgery~\cite{goodman83} between $M^1$ and $M^2$
by indentifying $C \in \partial M^1$ with $C^\prime \in \partial M^2$,
obtaining a new closed $3$-manifold $M$. Also, gluing $X^1$ and $X^2$
we can get a new vector field $X$ on $M$.
\end{description}

The nonwandering set of $X$ is contained in the union of the
\emph{sectional hyperbolic attracting set} with no Lorenz-like
singularities $A^* := W^{u}(\sigma_2) \cap A$ together with the
\emph{repelling singularity} $\sigma_1$; see \cite[Theorem
A]{Morales07} for more details.

Note that $\sigma_2$ is \emph{not active} and not Lorenz-like.

In addition, since $A^*$ is singular hyperbolic (because it is
sectional hyperbolic in a $3$-manifold) with $\sigma_2$ a hyperbolic
fixed saddle singularity, and $\sigma_1$ is a hyperbolic fixed source
whose unstable manifold is transverse to the stable manifold of $A^*$,
then any $C^r$ vector field $Y$ which is sufficiently $C^r$ close to
$X$ exhibits the same features, for any given $r\ge1$.
\end{example}

Examples~\ref{ex:gLorenzNeq} and~\ref{ex:SHnoLsing} are
sectional hyperbolic attracting sets which are \emph{not
  multisingular hyperbolic} according to the Definition~\ref{def:MSH}
of Crovisier et al.

Moreover, these examples are $C^r$ robust for any $r\ge1$, that is,
any $C^r$ small perturbation of the given vector fields exhibits a
non-multisingular hyperbolic attracting set. In addition, both
classes of vector fields are star vector fields. This completes the
proof of Theorem~\ref{mthm:nonmsh}.

\subsection{Proof of Theorem~\ref{mthm:nonsec3d}}
\label{sec:proof-theorem-refmth-1}

Next we present two examples of non-sectional hyperbolic attractors
which are MNUSE.

\begin{example}[Geometric Lorenz-like attractor with non-hyperbolic
  periodic orbit]\label{ex:notsinghyp}

  We construct a geometric Lorenz flow  in such a
  way to obtain, as the quotient over the stable leaves of the
  Poincar\'e first return map to the global cross-section of a vector
  field $G_0$, the following adaptation of the ``intermittent''
  Manneville map~\cite{ManPom80} 
  into a local homeomorphism of the circle; see
  Figure~\ref{fig:L1drep}: consider $I=[-1,1]$ and the map $f:I\to I$
  \[
    x\mapsto f(x):=
    \begin{cases}
      2\sqrt{x}-1 & \text{if $x\ge0$},
      \\
      1- 2\sqrt{|x|} & \text{otherwise}.
    \end{cases}
  \]

  \begin{figure}[htpb]
    \centering \includegraphics[width=3.5cm]{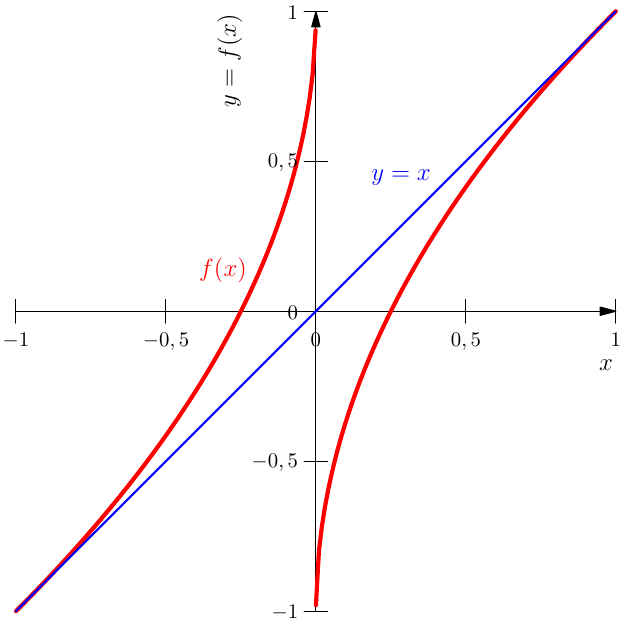}
    \qquad
    \includegraphics[width=4cm]{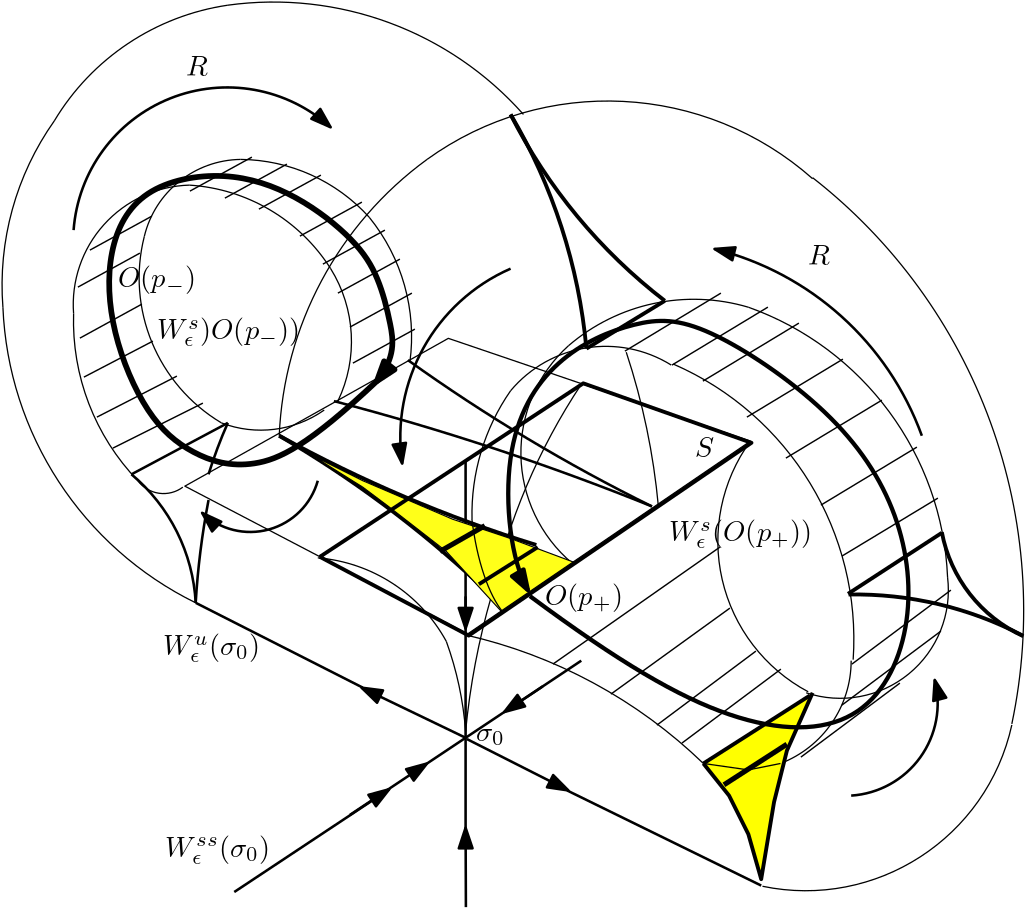}
    \caption{\label{fig:L1drep}Lorenz one-dimensional transformation
      $f$ with repelling fixed points at the extremes of the interval
      on the left; and the geometric Lorenz construction with this map
      as the quotient over the contracting invariant foliation on the
      cross-section $S$, with two corresponding periodic saddle-type
      periodic orbits $\SO(p_{\pm})$.}
  \end{figure}

  This construction is standard and the vector field can be made of
  class $C^\infty$, see e.g.~\cite[Chap. 3, Sec. 3.3]{AraPac2010}.
  
  Following the standard construction described in~\cite[Chap. 7,
  Sec. 3.4]{AraPac2010} to obtain a physical measure for this
  attractor when the vector field is of class $C^2$, we can show that
  the maximal invariant subset $\Lambda$ satisfies the MNUSE condition
  while not being sectional hyperbolic; see~\cite[Example 8]{ArSal25}
  for details. From \cite[Theorem E]{ArSal25}, these attractors admit
  a unique physical/SRB measure due to transitivity.
\end{example}

\begin{example}[Geometric Lorenz-like attractor with non-hyperbolic
  equilibrium]
  \label{ex:nonhypeq}
  Recently, Bruin and Farias~\cite{bruin2023mixing} construct
  (similarly to the previous example) a geometric Lorenz-like
  attractor of a $C^\infty$ flow with a neutral equilibrium replacing
  the hyperbolic Lorenz-like equilibrium from the classical
  (geometrical) Lorenz attractor, which is not hyperbolic. The authors
  show that there exists a unique physical/SRB measure, ensuring that
  the attractor satisfies MNUSE.
\end{example}

Both Examples~\ref{ex:notsinghyp} and~\ref{ex:nonhypeq} are MNUSE
attracting sets which are neither sectional hyperbolic (SH), nor
multisingular hyperbolic (MSH), nor asymptotically
sectional hyperbolic (ASH). Moreover, they contain a non-hyperbolic
periodic orbit and, thus, do not satisfy the Hyperbolic
Lemma~\ref{le:hyplemma}.

\begin{example}[Rovella attractors]\label{ex:Rovellalike}
  Another class of partially hyperbolic and mostly asymptotically
  sectional expanding attractors sets are the Rovella
  attractors~\cite{Ro93} (also known as contracting Lorenz attractors)
  which recently have been shown to be asymptotically
  sectional hyperbolic by San Martin and Vivas~\cite{SmartinVivas20}.

  It was shown by Metzger~\cite{mtz001} that these attractors admit a
  physical/SRB probability measure whose basin covers the trapping
  region except a zero volume subset and, moreover, exhibit slow
  recurrence to the equilibrium at the origin. For a short
  presentation of the features of this attractor family,
  see~\cite{araJSP21}.
  
  This is an example of a vector field of class $C^r, r\ge3$,
  with an ASH attractor containing equilibria which is not
  Lorenz-like, and admits an invariant ergodic physical probability
  measure. Hence, this attractor is ASH but neither SH nor MSH.

  We note that the $C^3$ condition is used in the construction of this
  family of vector fields in~\cite{Ro93} to better estimate the return
  times to the singularity of typical trajectories using techniques
  from one-dimensional dynamics.
\end{example}

This completes the proof of Theorem~\ref{mthm:nonsec3d}.






\section{Construction of new examples of  ASH attractors}
\label{sec:multid-exampl}

Here we construct examples proving items (1)-(3) of
Theorem~~\ref{mthm:nonsec}.

We consider a ``solenoid'' constructed over a uniformly expanding map
$g:\TT\to\TT$ of the $k$-dimensional torus $\TT$, for some
$k\ge2$. That is, let $\DD$ be the unit disk on $\RR^2$ and consider a
smooth embedding $F_0:N\circlearrowleft$ of $N=\TT\times\DD$ into
itself, which preserves and contracts the foliation
$\F^s=\big\{\{z\}\times\DD: z\in\TT\big\}$. We will write $E^s$ for
the tangent bundle to the leaves of this foliation. The natural
projection $\pi:N\to\TT$ on the first factor \emph{smoothly
  conjugates} $F_0$ to $g$: $\pi\circ F_0=g\circ\pi$ --- we can assume
that $\pi$ is the projection associated to a tubular neighborhood of
$F_0(N\times\{0\})$. We assume also that
$F_0$ admits a fixed point $p_0$ and that
$\lambda_1^{-1}\le\|(Dg)^{-1}\|\le\lambda_0^{-1}$ for some fixed
$\lambda_1>\lambda_0>1$. We fix $p\in N\times\{0\}$ such that
$F_0(p)=p_0$ and $p\neq p_0$.

\subsection{ASH attractor, with non sectionally
  hyperbolic equilibria}
\label{sec:multid-asympt-sectio}

We start with $k=1$, that is, with the three-dimensional Smale
solenoid map.

\subsubsection{The suspension of the solenoid map}
\label{sec:suspens-soleno-map}

We further consider the constant vector field $X:=(0,1)$ on
$M_0=N\times[0,1]$, which defines a transition map from
$\Sigma_\epsilon=N\times\{\epsilon\}$ to
$\Sigma_{1-\epsilon}=N\times\{1-\epsilon\}$ for some fixed small
$\epsilon>0$, which is the identity in the first coordinate when
restricted to $\Sigma_\epsilon$. Next we modify this field on the
cylinder $\cC=U\times\DD\times[0,1]$ around the periodic orbit of the
point $p_1=(p,0)\in N\times\{0\}$, where
$U=B_{\delta_0}(p)\times\{0\}$ is a small neighborhood of $p_1$ in
$\TT$, in such a way as to create both two equilibria
$\sigma_1,\sigma_2$, with either $k$ expanding and $3$ contracting
eigenvalues, or $1$ expanding and $k+2$ contracting eigenvalues, as
follows.  We fix $k=1$, so that $\TT=\sS^1$ and $U$ is an interval;
and ignore the stable foliation along $\DD$ in the next arguments.

\begin{figure}[htbp]
\includegraphics[width=7cm,height=6cm]{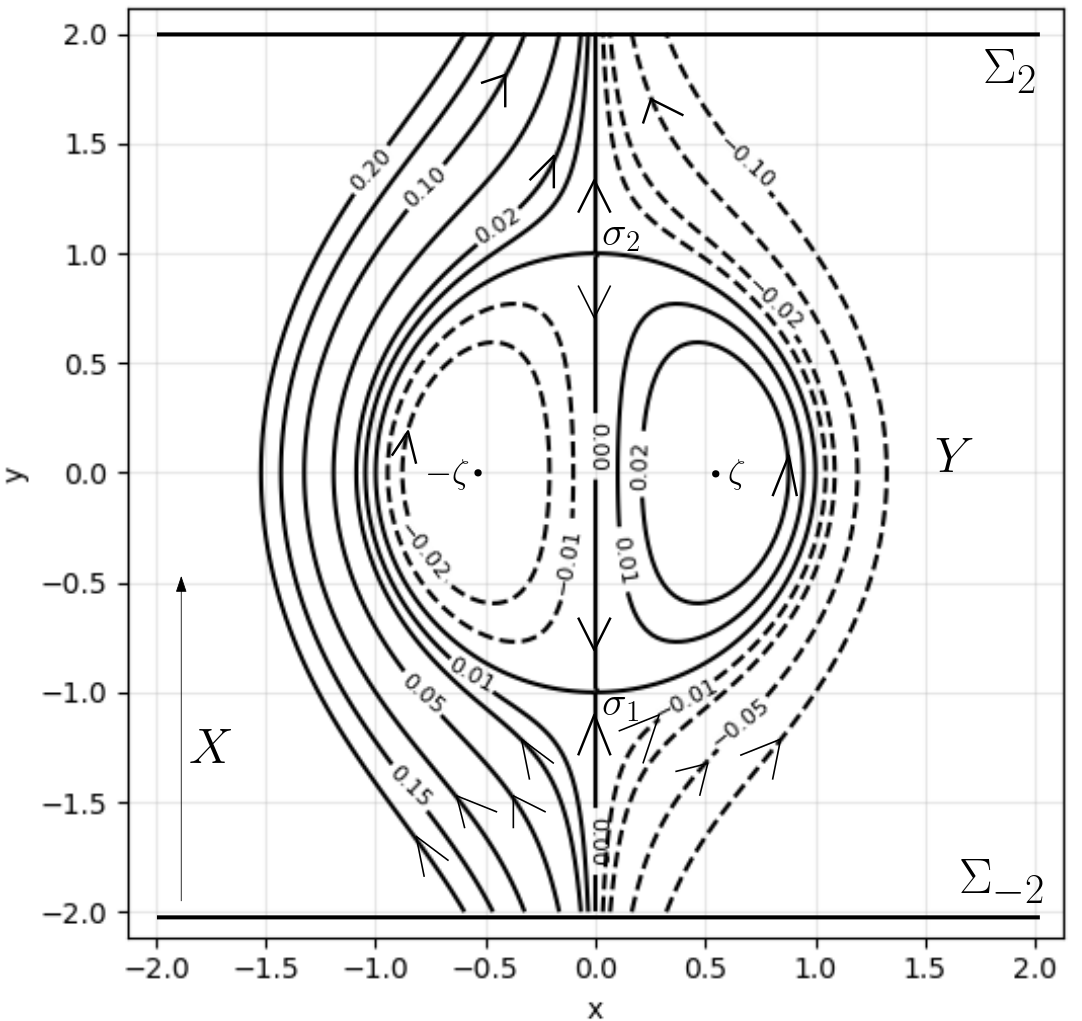}
\caption{A sketch of the vector field $Y$ with its Poincar\'e map from
  $\Sigma_{-2}$ to $\Sigma_2$ compared with the vector field $X$ in
  the square $[-2,2]\times[-2,2]$.}
\label{fig:H0}
\end{figure}
  
We consider the vector field $Y_0$ on the rectangle
$C_0=[-3,3]\times[-2,2]$ depicted in Figure~\ref{fig:H0} obtained from
the level curves of
$H(x,y):=x\big(1-(x^2+y^2)\xi_0(x)-2(1-\xi(x))\big)/10$, where
$\xi:\RR\to[0,1]$ is a smooth bump function $\xi_0:\RR\to[0,1]$ so
that $\xi_0\mid_{[-2,2]}\equiv1$,
$\xi_0\mid_{\RR\setminus[-3,3]}\equiv0$, $\xi_0(-x)=\xi_0(x)$ and
$\xi_0\mid_{\RR^+}$ is decreasing.

The field $Y_0$ is the  Hamiltonian vector field
$Y_0:= (H^\prime_y, -H^\prime_x)$ in the plane.
This ensures that \emph{$Y_0$ is conservative restricted to $C_0$} ---
recall that we are ignoring the uniformly contracting directions $E^s$
tangent to the foliation $\cF^s$.

We note that for $x\ge3$ we have $Y=X$.  Moreover, for $x\le2$ we can
explicitly write
 \begin{align*}
   Y_0(x,y)
   &=
     \left(-\frac{x y}5 , \frac{3x^2+y^2-1}{10}\right)
     \quad\&\quad
     DY_0(x,y)=
     \begin{pmatrix}
       -y/5 & -x/5 \\ 3x/5 & y/5
     \end{pmatrix}
     .
 \end{align*}

\subsubsection{Poincar\'e transition map is the identity}
\label{sec:poincare-transit-map}

The symmetry of the Hamiltonian ensures that the level curves of $H$
with non-zero values are symmetric with respect to the transformation
$S:(x,y)\mapsto(-x,y)$, i.e., $S(H^{-1}(\{\zeta\}))=H^{-1}(\{-\zeta\})$ for
$\zeta\neq0$ and these level curves are trajectories of the flow with
positive speed in the $y$ direction. 

\begin{claim}\label{cl:trId}
  The transition Poincar\'e map from $\Sigma_{-2}=[-3,3]\times\{-2\}$
  to $\Sigma_2=[-3,3]\times\{2\}$ is the identity away from the point
  $(0,-2)$
\end{claim}

\begin{remark}[time to cross]
  \label{rmk:crosstime}
  For $2\le|x|\le3$ the flow on $C_0$ has a vertical speed along the
  positive direction of the $y$-axis of at least
  $(2^2-1)/10=3/10$. Hence, starting from $(x,-2)$ the flow arrives at
  $(x,2)$ after a time of $t(x)\le 4\cdot 10/3\le16$.
\end{remark}

\subsubsection{Non-sectional hyperbolic equilibria}
\label{sec:non-section-hyperb}
  
The eigenspace of one of the contracting (expanding) eigenvalues of the
equilibria $\sigma_1,\sigma_2$ lies along the vertical direction (the
direction of $X$), \emph{the other two-dimensional contracting
  directions still lie on the direction of $\DD$ (ignored in the
  pictures)}, and the remaining expanding/contracting eigenspaces are
transversal to the vertical $X$ direction; see
Figure~\ref{fig:H0}. There are also a pair of fixed elliptic
equilibria represented by $\pm\zeta$.

We have $\sigma_i=(0,(-1)^i), i=1,2$ and $\zeta=(\sqrt3/3,0)$
so that
\begin{align}\label{eq:DY0}
  DY_0(\sigma_i)=
  \begin{bmatrix}
    (-1)^{i+1}/5
    & 0
    \\ 0
    & (-1)^i/5
  \end{bmatrix}
      \quad\&\quad
      DY_0(\pm\zeta)
      =\pm
      \begin{bmatrix}
        0
        & -\sqrt3/15
        \\
        \sqrt3/25
        &0
      \end{bmatrix}.
\end{align}
This shows that $\sigma_i$ are hyperbolic saddles which are \emph{not
  sectionally hyperbolic}: neither sectionally expanding, nor
sectionally contracting, since their traces vanish.

\subsubsection{Partial hyperbolic attractor}
\label{sec:partialy-hyperb}

After rescaling, we assume that $Y_0$ is defined in the initial cylinder
$\cC$, by setting the coordinates corresponding to the factor $\DD$
equal to zero.  We also assume that the standard inner product
satisfies $\langle Y_0, X\rangle>0$ on the Poincar\'e sections
$\Sigma_{\epsilon}\cup\Sigma_{1-\epsilon}$ corresponding to
$\Sigma_{-2}\cup\Sigma_2$; and take a $C^\infty$ bump function
$\psi:[0,1]\circlearrowleft$ so that
$\psi\mid_{[\epsilon/2,1-\epsilon/2]}\equiv0$ and
$\psi\mid_{[0,\epsilon/3]\cup[1-\epsilon/3,1]}\equiv1$. Then, we
define the vector field
\begin{align}\label{eq:attached}
  G_0(x,u):=
  \psi(u)\cdot X+(1-\psi(u))\cdot Y_0(x,u), \quad (x,u)\in M_0
\end{align}
which generates a smooth transition map $L$ from
$\Sigma_0^*=(N\setminus\{p\})\times\{0\}$ to
$\Sigma_1=N\times\{1\}$. \emph{Since the Poincar\'e transition maps of
  both $X$ and $Y$ are the identity, then $L=Id$.}

Together with the identification $(x,0)\sim(F_0(x),1), x\in N$ we
obtain a smooth parallelizable $4$-manifold $M=M_0/\sim$ where $G_0$
induces a $C^\infty$ vector field which we denote by the same
letter. We write $(\phi_t)_{t\in\RR}$ for the induced flow.  We also
have an attracting subset $\Lambda=\cap_{t\ge0}\phi_t(M)$ with $M$ as
topological basin of attraction.

\emph{The Poincar\'e first return map of this vector field
  $P:\Sigma_0^*\to\Sigma_0$ coincides with
  $F_0\mid_{N\setminus\{p\}}$.} In particular,
$\Lambda_0:= \cap_{n\in\ZZ_0^+}F_0^n(N)$ has an open and dense subset
of dense trajectories, and so the flow $\phi_t$ of $G$ is transitive
on $\Lambda$. Thus, $\Lambda$ is an attractor.

Since $\Lambda_0$ admits a $DF_0$-invariant hyperbolic splitting
$T_{\Lambda_0}N=E^s_{\Lambda_0}\oplus E^u_{\Lambda_0}$, and we may
assume without loss of generality that the contracting rate along
$E^s$ is stronger than the contracting eigenvalues of
$\sigma_1, \sigma_2$, then setting
\begin{align*}
  E^s_{(w,t)}:=D\phi_t(E^s_w) \quad \& \quad
  E^{cu}_{(w,t)}:=D\phi_t(E^u_w)\oplus \RR\cdot
  X, \quad w\in\Lambda_0, t\in[0,1);
\end{align*}
we obtain a $D\phi_t$-invariant and continuous splitting
$T_\Lambda M = E^s \oplus E^{cu}$ which is partially hyperbolic.

\subsubsection{Asymptotical sectional expansion}
\label{sec:asympt-section-expan}

Since the area along any $2$-plane of $T(U\times[0,1])$ is preserved,
we get $\psi^{cu}(w)=\log\|\wedge^2 (D\phi_1(w)\mid E^c_w)^{-1}\|=0$.

Hence, given $w\in\Sigma_0^*$ whose future trajectory visits $\cC$
infinitely many times, there exist sequences $n_i<m_i<n_{i+1}$ of
iterates so that $n_0=0$ and, for $j=n_i,\dots,m_i-1$, we have
$f^jw=\phi_1^j(w)\in\cC$; and $f^jw\in M\setminus\cC$ for
$j=m_i,\dots,n_{i+1}-1$. The previous argument shows that
\begin{align}\label{eq:cancel}
  \sum\nolimits_{j=n_i}^{m_i-1}\psi^{cu}(f^jw)=0.
\end{align}
Since on $M\setminus\cC$ the time-$1$ map on $\Sigma_0^*$ coincides
with $P$, then for $f^i(w)\in M\setminus\cC$ we can assume that
$f^i(w)\in\Sigma_0^*$ and obtain
\begin{align}\label{eq:NUSE00}
  \sum\nolimits_{j=m_i}^{n_{i+1}-1}\psi^{cu}(f^jw)
  \le
  -(n_{i+1}-m_i)\log\lambda_0.
\end{align}
Thus, we can write (grossly underestimating the number of iterates
outside of $\cC$ from $0$ to $n_{i+1}$ by $i$)
\begin{align*}
  \frac1{n_{i+1}}\sum\nolimits_{j=0}^{n_{i+1}-1}\psi^{cu}(f^jw)
  \le
  \frac{i-n_i}{n_{i+1}}\log\lambda_0. 
\end{align*}
We note that close to the stable manifold $W^s(\sigma_1)$ of
$\sigma_1$ in $\cC$ the time-$1$ map takes a potentially unbounded
amount of iterates to cross $\cC$ from bottom to top. Moreover, given
$\epsilon>0$ we can find $\delta_0>0$ so that any $w\in\Sigma_0^*$,
which visits a small neighborhood $B_\delta(p)$ away from $p$
infinitely many times, satisfies $i/n_i\le\epsilon$ as
$i\to\infty$. Thus
$\frac{i-n_i}{n_{i+1}}\log\lambda_0<-(1-\epsilon_0)\log\lambda_0$ for
all large enough $i$.

This shows that the flow satisfies the (wNU2SE) condition for
all trajectories which visit $B_{\delta_0}(p)\setminus\{p\}$ infinitely
many times or just a finite number of times. All trajectories of
$w\in\Lambda$ which do not converge to $\sigma_1$
(i.e. $w\in\Lambda\setminus W^s(\sigma_1)$) as well as all points of
the stable leaf $W^s(w)$ in $M$ satisfy this.

Since points $w$ whose trajectories do not pass through the point $p$
form a full Lebesgue measure subset, we have shown that the attractor
$\Lambda$ satisfies wNU2SE, and so is ASH from
Remark~\ref{rmk:NUSE-wNU2SE}(2).

\begin{remark}\label{rmk:wNU2SElarge}
  Moreover, this also holds for all trajectories of the ambient space
  $M$ which do not converge to the singularities.  Hence, as long as
  the flow is of class $C^2$, it admits a unique physical/SRB measure
  $\mu$ from Theorem~\ref{mthm:hdASH} with full basin:
  $\leb(M\setminus B(\mu))=0$.
\end{remark}

\begin{remark}[non-robustness]
  The properties of the flow $Y_0$ are clearly not robust: the
  perturbation $\wh{Y}(x,y):=(H'_y(x,y)-x/10, -H'_x(x,y))$ has
  trajectories sketched in Figure~\ref{fig:HamFlowPert}.
  \begin{figure}[htpb]
    \centering
    \includegraphics[width=5cm,height=4cm]{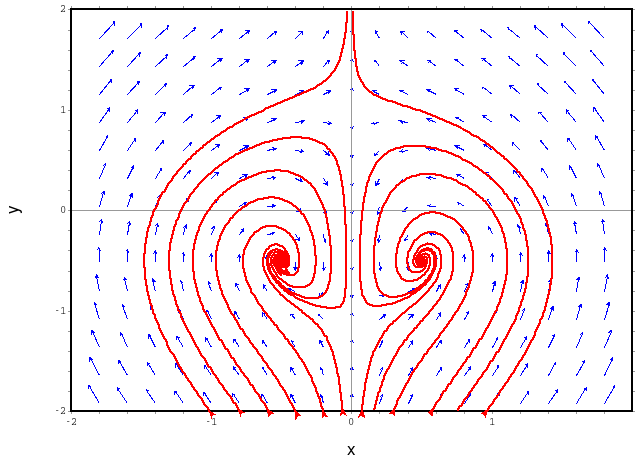}
    \caption{\label{fig:HamFlowPert} Sketch of the flow of the vector
      field $\wh{Y}$ with some trajectories, showing that trajectories
      starting close to $(0,-2)$ (that is, close to $p$ in the
      original suspension flow) will fall into sinks.}
  \end{figure}
\end{remark}

\subsection{ASH attractor with mixed sectionally
  hyperbolic equilibria}
\label{sec:multid-ash-attract}

We now modify the previous example to get sectional hyperbolic
equilibria.  We keep the following symmetry relations from
$Y_0(x,y)=\big(Y_0^1(x,y),Y_0^2(x,y)\big)$
\begin{align}\label{eq:symmetry}
  Y_0^1(-x,y)=-Y_0^1(x,y)=Y_0^1(x,-y)
  \quad\&\quad
  Y_0^2(\pm x,\pm y)=Y_0^2(x,y)
\end{align}
by setting
$$Y_1(x,y) := \big(H^\prime_y(x,y), -2\cdot H^\prime_x(x,y)\big).$$
It is straightforward to check that
\begin{itemize}
\item $\diver (Y_1)=-H^{\prime\prime}_{yx}=y/5$; 
\item the $y$-axis $\{x=0\}$ is still invariant; and
\item the points $\sigma_1,\sigma_2$ and $\pm\zeta$ are still
  equilibria;
\item for the equlibria $\sigma_1, \sigma_2$ we obtain the same
  properties as in~\eqref{eq:DY0} with the second row multiplied by
$2$.
\end{itemize}
This provides sectional hyperbolicity: $\sigma_1$ becomes a
sectionally contracting (``Rovella-like'') singularity along
$E^{cu}_{\sigma_1}$; while $\sigma_2$ becomes a sectionally expanding
(``Lorenz-like'') singularity along $E^{cu}_{\sigma_1}$.

\begin{remark}[crossing time]
  \label{rmk:crosstime2}
  Again, for $2\le|x|\le3$ the flow $Y_1$ on $C_0$ has a vertical speed
  along the positive direction of the $y$-axis of at least
  $(2^2-1)/5=3/5$. Hence, starting from $(x,-2)$ the flow arrives at
  $(x,2)$ after a time of $t(x)\le4\cdot 5/3\le8$.
\end{remark}



\subsubsection{Poincar\'e transition map is the identity}
\label{sec:poincare-transit-map-1}

This vector field also satisfies Claim~\ref{cl:trId} since we have the
following properties of the trajectories of $Y_1$. 
\begin{lemma}[symmetric solutions]
  \label{le:symmetric}
  For any $\epsilon>0$, the trajectories
  $\big(\gamma(t)\big)_{t\in(-t_0,t_0)}$ of a vector field $Y_1$ with
  $\gamma(t)=\big(x(t),y(t)\big)$ satisfying~\eqref{eq:symmetry} are
  such that
  \begin{enumerate}[(a)]
  \item $\wt\gamma(t):=-\gamma(t), t\in(-t_0,t_0)$ is a
    trajectory of $-Y_1$; and
  \item $\wh{\gamma}(t):=(x(t),-y(t)), t\in(-t_0,t_0)$ is also a
    trajectory of $-Y_1$.
  \end{enumerate}  
\end{lemma}

\begin{proof}
  Just observe that since $x'(t)=Y_1^1\big(\gamma(t)\big)$ and
  $y'(t)=Y_1^2\big(\gamma(t)\big)$ 
  \begin{align*}
    \wt\gamma^\prime(t)
    &=
      -\gamma'(t)
      =
      -Y_1\big(\gamma(t)\big)
      =
      -Y_1\big(-\gamma(t)\big)
    =
    -Y_1\big(\wt\gamma(t)\big); \qand
    \\
    \wh{\gamma}^\prime(t)
    &=
      \big(x'(t),-y'(t)\big)
      =
      \Big(Y_1^1\big(x(t),y(t)\big) ,
      -Y_1^2\big(x(t),y(t)\big)\Big)
    \\
    &=
      \Big( - Y_1^1\big(x(t),-y(t)\big) ,
      -Y_1^2\big(x(t),-y(t)\big)\Big)
      =
      -Y_1(\wh\gamma(t));
  \end{align*}
  for each $-t_0<t<t_0$.
\end{proof}

To prove the claim, we note that from Lemma~\ref{le:symmetric}, for
each trajectory $\gamma(t)$ starting at $(x_0,-2)$ with $x_0\neq0$ and
crossing $\cC$ to a point $(x_1,2)$, there corresponds a trajectory
$\wh\gamma(t)$ of $-Y_1$, which starts at $(x_0,2)$ and crosses to the
points $(x_1,-2)$; see Figure~\ref{fig:symmetry}.

\begin{figure}
  \includegraphics[height=5cm]{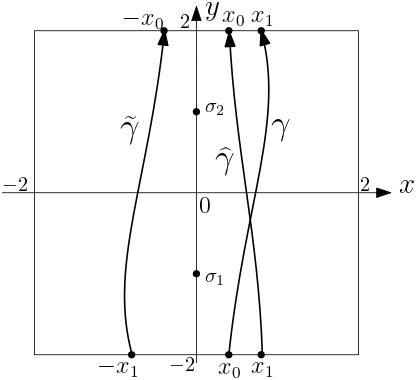}
  \caption{\label{fig:symmetry} Sketch of the trajectories $\gamma$,
    $\wt\gamma$ and $\wh\gamma$ of $Y_1$ if $0<x_0<x_1$.}
\end{figure}

We claim that $x_0=x_1$.  Arguing by contradiction, if $0<x_0<x_1$, we
have a pair of trajectories of a flow starting at $(x_0,-2)$ and
$(x_1,-2)$; and crossing to the points $(x_1,2)$ and $(x_0,2)$. Since
the order was exchanged, there must be an intersection of the
trajectories. This contradiction proves the claim.

Property (a) from Lemma~\ref{le:symmetric} ensures that the same
argument holds for $x_0<0$.
Since all trajectories starting at $(x_0,-2)$ with $x_0\neq0$ cross to
$(x_0,2)$, we have proved Claim~\ref{cl:trId}.

\begin{claim}[heteroclinic connection]
  We have the heteroclinic connection $W^u(\sigma_1)=W^s(\sigma_2)$.
\end{claim}

Indeed, since all points $(x_0,-2)$ with $x_0\neq0$ cross to a point
$(x_0,2)$, then we keep the heteroclinic connection --- for otherwise,
some trajectory in $W^u(\sigma_1)$ would cross to a point $(x_1,2)$
for some $x_1\neq 0$ (since the $y$-axis
is still invariant) and this contradicts the symmetry of solutions.

\subsubsection{Asymptotically sectional expansion}
\label{sec:asympt-section-expan-1}

Considering the flow $G_1$ obtained from $Y_1$ as
in~\eqref{eq:attached}, the symmetry of the flow on $\cC$ enables us
to use similar arguments as in
Subsection~\ref{sec:asympt-section-expan}: given $w\in\Sigma_0^*$
whose trajectory crosses $\cC$ infinitely many times, we consider the
sequence $n_i<m_i<n_{i+1}$ of iterates bounding visits of the
trajectory of $w$ in $\cC$; and note that the regions
\begin{align*}
  \{\wh{w} \in \cC: \psi^{cu}(\wh{w})=\diver(Y_1)(\wh{w})>0\}
  \qand
  \{\wh{w} \in \cC: \psi^{cu}(\wh{w})=\diver(Y_1)(\wh{w})<0\}
\end{align*}
are symmetric with the same size: they correspond to the upper half
($y>0$) and lower half ($y<0$) of the cylinder $C_0$.  Both are
traversed by each trajectory through $\cC$ in a symmetric way using
the same number of iterates modulo a finite difference. Hence, we can
write~\eqref{eq:cancel}. We also keep~\eqref{eq:NUSE00}.  So wNU2SE
follows. In addition, Remark~\ref{rmk:wNU2SElarge} also
holds.

\begin{remark}[physical/SRB]\label{rmk:phSRBC2}
  From Lemma~\ref{le:symmetric} we have that the attracting set admits
  dense trajectories, and so we have a unique physical measure from
  Theorem~\ref{mthm:hdASH} as long as the flow is of class $C^2$.
\end{remark}
\subsection{Higher codimensional ASH attractor}
\label{sec:multidimesional-ash}

Now we restart with $k=\ell+1$ the previous suspension flow
construction, for any fixed $\ell\ge1$, and adapt the two-dimensional
vector field $Y_1$ by considering the vector field $Y_2$ in the
cylinder $C_1=[-3,3]\times B^\ell\times[-2,2]$, where $\|y\|_2$ is the
Euclidean norm; $B^\ell=\{y\in\RR^\ell:\|y\|_2\le 3\}$ is the closed
unit ball with radius $3$ centered at the origin, and $Y_2$ is given
by
\begin{align}\label{eq:Y2}
  Y_2(x,y,z):=(H'_z(x,z),\omega\cdot\xi_1(\|y\|_2^2)\cdot y,-2H'_x(x,z)),
  \quad (x,y,z)\in C_1
\end{align}
for some fixed $\omega>0$. Here, $\xi_1:\RR\to[0,1]$ is a smooth bump
function such that $\xi_1\mid_{\RR^+}$ is decreasing;
$\xi_1(-t)=\xi_0(t)$ for all $t\in\RR$; $\xi_1\mid_{[-4,4]}\equiv1$
and $\xi_1\mid_{\RR\setminus[-9,9]}\equiv0$.

\begin{remark}[explicit solution and crossing time]
  \label{rmk:crosstime3}
  We can explicitly solve for $y$ with initial condition
  $y_0\in B^\ell$ with $\|y\|_2<2$: $y(t)=y_0e^{\omega t}$ for
  $0\le t\le \omega^{-1}\log(2/\|y_0\|_2)$.

  Since the $x,z$-components of $Y_2$ coincide with the components of
  $Y_1$, the crossing time of the flow from $\Sigma_{-2}$ to
  $\Sigma_2$ is again at most $8$ for $2\le |x|\le3$.
\end{remark}

\subsubsection{(Sectional-)Hyperbolicity of equilibria}
\label{sec:section-hyperb-equil-1}

It is easy to see that the equilibria are
$\sigma_i=(0,0,(-1)^i), i=1,2$ and $\zeta=(\sqrt3/3,0,0)$; where
$\zeta$ is non-hyperbolic, $\sigma_{1,2}$ are both hyperbolic of
saddle-type, with $\sigma_2$ generalized Lorenz-like and $\sigma_1$
generalized Rovella-like.

Moreover, we have for $i=1,2$
\begin{align}
  \label{eq:cuJac}
  \psi^{cu}(\sigma_i)
  =
    (-1)^{i+1}/5 \qand
    \log J^{cu}f(\sigma_i)
  =
    \omega + (-1)^{i+1}/5.
\end{align}

\subsubsection{Sectional hyperbolicity along $C_1$}
\label{sec:section-hyperb-along}

We can write, using the product structure of $C_1$, for 
$u=(x,y,z)\in C_1$ so that $\varrho^2=x^2+\|y\|_2^2\le4$
\begin{align*}
  DY_2(x,y,z)=
  \begin{pmatrix}
    -z/5 & 0 & -x/5
    \\
    0 & \omega & 0
    \\
    6x/5 & 0 & 2z/5
  \end{pmatrix};
\end{align*}
and so the subbundles $E^1_u:=\RR\times 0^\ell\times\RR$ and
$E^2_u=0\times\RR^\ell\times 0$ are $D\phi_t$-invariant under the flow
$\phi_t$ of $Y_2$ inside the subcylinder $\varrho\le4$. While $u$ and
$\phi_{[0,t]}(u)$ are in this subcylinder, for some $t>0$, we have the
following domination property
\begin{align*}
  \|D\phi_t\mid E^1_u\|
  \le
  e^{12t/5}
  <
  e^{\omega t }
  =
  \|\big(D\phi_t\mid E^2_u)^{-1}\|;
\end{align*}
as long as $\omega>12/5$.
This ensures that the least expansion along any $2$-subspace
by $D\phi_1\mid E^c_u$ at $u=(x,y,z)\in C_1$ is achieved along the
$E^1_u$-subbundle, that is
\begin{align*}
  \big\|\wedge^2\big( D\phi_1\mid E^c_u \big)^{-1}\big\|
  =
  \big\|\wedge^2\big( D\phi_1\mid E^1_u \big)^{-1}\big\|
  =
  \big\|\wedge^2\big(D\wh{\phi}_1\mid E^c_{(x,z)}\big)^{-1}\big\|;
\end{align*}
where $(\wh{\phi}_t)_{t\in\RR}$ is the flow of the vector field $Y_1$
from the previous subsection.

The symmetry of the trajectories in the $(x,z)$ variables together
with the above choice of $\omega$, ensures that, for each trajectory
crossing $C_1$, the portion of the trajectory covering the region
$z<0$ contributes to the sum~\eqref{eq:cancel} by the same amount, but
of opposite sign, as the portion of the same trajectory covering the
region $z>0$.  Hence, we reobtain~\eqref{eq:cancel}.

  \subsubsection{Asymptotical sectional expansion}
  \label{sec:asympt-section-expan-6}
  
  We observe that we do not necessarily have the Poincar\'e map $P$
  coinciding with the original expanding map $g$, since now we do not
  have symmetry on the $y$ variable, although the transition map from
  $(x,y,-2)$ to $(x,\bar y,2)$ keeps the $x$-variable.

  By construction, \emph{the transition map expands the $y$-variable
    close to $0$}, but admits a contraction away from $0$ due to the
  use of the bump function in the definition of $Y_2$. Nevertheless,
  by Remark~\ref{rmk:crosstime3}, the crossing time of the possible
  contracting region is at most $2$, the contraction rate is bounded;
  and we may assume the value of $\lambda_0>1$ large enough so that
  the transition map of the flow of the vector field $G_3$, obtained
  from $Y_3$ by the same procedure as~\eqref{eq:attached}, is still
  uniformly expanding.  Therefore, we also keep~\eqref{eq:NUSE00} with
  $\wh{\lambda_0}>1$ in the place of $\lambda_0$.

  Thus, the same argument from the previous
  Subsections~\ref{sec:asympt-section-expan}
  and~\ref{sec:asympt-section-expan-1} provides the wNU2SE property,
  on all trajectories not converging to a singularity, as in
  Remark~\ref{rmk:wNU2SElarge}.
  
  Thus, $\Lambda$ becomes an attractor with $d_{cu}=\ell+2$ for
  any fixed positive integer $\ell \ge1$ and satisfying NUSE.

  \begin{remark}[physical/SRB]\label{rmk:phSRBC2a}
    From~\eqref{eq:cuJac}, we have condition (B) of
    Theorem~\ref{mthm:hdASH}, and so we can ensure existence and
    uniqueness of a physical/SRB measure for this flow, as long as the
    previous construction is at least $C^2$ smooth.
  \end{remark}
  
\section{Construction of non-uniformly sectional expanding
   attractors}
\label{sec:new-examples-non}

Here, we construct examples proving items (4)-(5) of
Theorem~~\ref{mthm:nonsec}.

We extend symmetrically the vector fields $Y_0, Y_1$ to
three-dimensional versions first.

\subsection{Higher codimensional NU2SE with non-sectional hyperbolic
  equilibria}
\label{sec:multid-ash-}

We now assume that $k=2$ and rotate the setup of Figure~\ref{fig:H0}
around the vertical axis: we set for
$(\varrho,\theta,z)\in[0,3]\times[0,2\pi]\times[-2,2]$
\begin{align*}
  Y_3(\varrho \cos\theta, \varrho \sin\theta, z) :=
  H^\prime_y(\varrho,z)\cdot(\cos\theta,\sin\theta,0)-
  H^\prime_x(\varrho,z)\cdot X;
\end{align*}
the corresponding symmetrized vector field from the plane Hamiltonian
vector field $Y_0$.

\begin{remark}[consequences of symmetry]
  \label{rmk:symmetry}
  The symmetry ensures that $Y_3$ also satisfies
  Claim~\ref{cl:trId}. Hence, the flow $\phi_t$ of $G_3$ induces a
  transition map $L$ from $\Sigma_0^*=(N\setminus\{p\})\times\{0\}$ to
  $\Sigma_1=N\times\{1\}$ \emph{which is the identity}, as in
  Subsection~\ref{sec:poincare-transit-map}.

  Moreover, every vertical plane containing the $z$-axis, i.e, with
  equation $ax+by=0$ for any pair $(a,b)\neq(0,0)$, is preserved by
  the flow along with its area.
\end{remark}

\subsubsection{Hyperbolic and non-sectional hyperbolic equilibria}
\label{sec:hyperb-non-section}

We can write more explicitly for $\varrho\le2$, since
$\varrho^2=x^2+y^2$ 
\begin{align*}
  Y_3(x,y,z)
  &=
    -\frac{\varrho z}5\cdot\frac{(x,y,0)}{\sqrt{x^2+y^2}}
    +\left(0,0,\frac{3\varrho^2+z^2-1}5\right)
    =
    \left(-\frac{xz}5,-\frac{yz}5,\frac{3x^2+3y^2+z^2-1}{10}\right)
\end{align*}
and it is now easy to calculate
\begin{align*}
  DY_3(x,y,z)
  =
  \begin{pmatrix}
    -z/5 & 0 & -x/5
    \\
    0 & -z/5 & -y/5
    \\
    3x/5 & 3y/5 & z/5
  \end{pmatrix}
                  \qand
                  \diver(Y_3)\equiv -z/5.
\end{align*}
It is easy to see that equilibria are given by the pair
$\sigma_i=(0,0,(-1)^i), i=1,2$, of hyperbolic saddles which are not
sectionally hyperbolic; and
$\zeta(\alpha):=(\sqrt3/3)\cdot(\cos\alpha,\sin\alpha,0),
\alpha\in[0,2\pi)$, a circle of elliptical fixed points.

We attach $Y_3$ to the suspension flow $X$ as in~\eqref{eq:attached}
obtaining a $(k+3)$-dimensional flow $G_3$ (recall that here $k=2$ and
the stable direction is two-dimensional).



\subsubsection{Asymptotic sectional expansion}
\label{sec:asympt-section-expan-3}

From Remark~\ref{rmk:symmetry} at any point $w\in C_0$ there exists a
$2$-plane whose area is preserved by $D\phi_t(w)=Df(w)$.

For a point $w\in(U\setminus\{p\})\times\{0\}$ close to $p$, the time
$\tau(w)$ needed to cross $\cC$ can be estimated as
$\tau(w)\le C\cdot|\log d(w,p)|$ for some constant $C>0$, since $p$
belongs to the stable manifold of the hyperbolic saddle equilibria
$\sigma_1$. This ensures that $\tau$ is $\leb$-integrable.

The exterior product $\|\wedge^2\big(D\phi_t\mid E^c_w\big)^{-1}\|$ is
bounded above by $\|D\phi_t\|^2$, and from the Linear Variational
Equation and the Gronwall's Inequality $\|D\phi_t(w_0)\|\le e^{t \|DY_3\|}$,
where $\|DY_3\|=\sup_{w\in \cC}\|DG_3(w)\|$ for any $w_0\in\cC$. Hence
we get
\begin{align}\label{eq:highest}
  \|DY_3\|\tau(w)
  \ge
  \sum\nolimits_{i=0}^{[\tau(w)]} \psi^{cu}(f^i(w)).
\end{align}
Arguing as in Subsection~\ref{sec:asympt-section-expan}, given
$w\in\Sigma_0^*$ whose future trajectory visits $\cC$ infinitely many
times, we consider the same sequences $n_i<m_i<n_{i+1}$ of iterates
marking the crossings of $\cC$. From the above arguments, we
keep~\eqref{eq:NUSE00} and use~\eqref{eq:highest} to
replace~\eqref{eq:cancel} by the following
\begin{align}
  \label{eq:cancel1}
  \sum\nolimits_{j=n_i}^{m_i-1}\psi^{cu}(f^i(w))
  \le
  \|DY_3\|\cdot\tau\big(f^{n_i}w\big).
\end{align}
Now we can estimate
\begin{align}
  \sum\nolimits_{j=0}^{n_{i+1}-1}
  &\psi^{cu}(f^i(w))
  \le
  -\log\lambda_0\sum\nolimits_{k=0}^i(n_{k+1}-m_k)
    +\|DY_3\|\sum\nolimits_{k=0}^i(m_k-n_k)\nonumber
  \\
  &=\label{eq:estima0}
    -\log\lambda_0\cdot\#\{0\le j<n_{i+1}: \phi_1^j(w)\in
    \Sigma_0^*\setminus U\}
    +\|DY_3\|\sum\nolimits_{k=0}^i\tau\big(f^{n_k}w\big).
\end{align}
Since on $M\setminus\cC$ the time-$1$ map on $\Sigma_0^*$ coincides
with $P$, then we can recount the iterates of $\phi_1^j(w)$ through
the iterates $P^k(w)$: we set $\tau\mid_{\Sigma_0^*\setminus U}\equiv1$ and
$\ell(i):=i+\sum_{k=0}^i (n_{k+1}-m_k)$ the lap number, note that
$n_{i+1} = \sum_{k=0}^{\ell(i)}\tau(P^kw)$ and
rewrite~\eqref{eq:estima0} as
\begin{align*}
  -\log\lambda_0
  &\cdot\#\{0\le k<\ell(i): P^k(w)\in
    \Sigma_0^*\setminus U\}
  \\
  &+\|DY_3\|\sum\{\tau\big(P^kw\big): 0\le k<\ell(i): P^k(w)\in U\}
  \\
  &=
    \sum\nolimits_{j=0}^{\ell(i)}
    \Big(
    \big(-\log\lambda_0\mathbf{1}_{\Sigma_0^*\setminus U}
    +\|DY_3\|\mathbf{1}_U \big)\cdot
    \tau\Big)\circ P^j (w).
\end{align*}
Hence, for $w\in (U\setminus\{p\})\times\{0\}$ we can write
\begin{align}
  \frac1{n_{i+1}}
  &\sum_{j=0}^{n_{i+1}-1} \psi^{cu}(f^jw)
    \le \nonumber
    -\frac{\log\lambda_0}{n_{i+1}}
    \sum_{k=0}^{\ell(i)}\big(\mathbf{1}_{\Sigma^*_0\setminus U}\big)(P^kw)
    +
    \frac{\|DY_3\|}{n_{i+1}}
    \sum_{k=0}^{\ell(i)}\big(\tau\mathbf{1}_{U}\big)(P^kw)
    \\
  &=\label{eq:medias}
    -\frac{\ell(i)\log\lambda_0}{n_{i+1}}
    \frac1{\ell(i)}\sum_{k=0}^{\ell(i)}
    \big(\mathbf{1}_{\Sigma^*_0\setminus U}\big)(P^kw)
    +
    \frac{\ell(i)\|DY_3\|}{n_{i+1}}\frac1{\ell(i)}
    \sum_{k=0}^{\ell(i)}\big(\tau\mathbf{1}_{U}\big)(P^kw),
\end{align}
and by ergodicity of $\leb_\Sigma$ with respect to $P$, since
$\tau\mid_{\Sigma_0^*\setminus U}\equiv1$ and
\begin{align}\label{eq:intau}
  \frac{n_{i+1}}{\ell(i)}
  =
  \frac1{\ell(i)}\sum\nolimits_{k=0}^{\ell(i)}\tau(P^kw)
  \xrightarrow[i\to+\infty]{}\leb(\tau)=\int\tau\,d\leb_{\Sigma},
  \quad \leb_\Sigma-\text{a.e. } w\in\Sigma_0^*;
\end{align}
we arrive at
\begin{align*}
  \limsup_{i\to\infty}\frac1{n_{i+1}}
  \sum\nolimits_{j=0}^{n_{i+1}-1}\psi^{cu}(f^jw)
  \le
  -\frac{\log(\lambda_0)}{\leb(\tau)}
  (1-\leb_\Sigma(U))
  +
  \frac{\|DY_3\|}{\leb(\tau)}\leb_\Sigma(\tau\mathbf{1}_{U}).
\end{align*}
Finally, since $\tau\in L^1(\leb_\Sigma)$ and $U=B_\epsilon(p)$, we
can make $\leb_\Sigma(\tau\mathbf{1}_{U})=\int_U \tau\,d\leb_\Sigma$
as close to zero as needed.

Since trajectories eventually returning to a full
$\leb_\Sigma$-measure subset of $\Sigma_0^*$ form a full volume subset
of the ambient manifold $M$, we can thus conclude NU2SE as long as $U$
is small enough. We note that for this it is enough that $Y_3$ be of
class $C^1$. From Remark~\ref{rmk:NUSE-wNU2SE} we have ASH and also
MNUSE.

\subsubsection{Slow recurrence and physical measure}
\label{sec:slow-recurr-equilibr}

To obtain a physical/SRB measure with full basin, it is enough to
obtain slow recurrence according to Theorem~\ref{thm:discretefabv}
after ensuring that $Y_3$ is of class $C^2$ in the previous
construction. We explore the invariance and ergodicity of
$\leb_\Sigma$ with respect to the Poincar\'e first return map $P$, the
integrability of the Poincar\'e first return time, together with the
symmetry of the flow on the cylinder $\cC$.

We use the equivalence between the SR condition~\eqref{eq:SR} and its
continuous version: on a positive volume subset of points, for every
$\epsilon>0$, we can find $r>0$ so that
\begin{align}
  \label{eq:SSR}
  \limsup_{T\nearrow\infty}\frac1T\int_0^T-\log d_r
  \big(\phi_t(x),\sing_\Lambda(G)\big) \, dt <\epsilon;
\end{align}
see~\cite[Theorem C]{ArSal25} for the proof of the stated equivalence.

In what follows we write
$\Delta_r(x):=-\log d_r \big(\phi_t(x),\{\sigma_1,\sigma_2\}\big)$ and
consider trajectories starting at a point $x\in\Sigma_0^*$ on a subset
with full $\leb_\Sigma$-measure, and claim that we can find a constant
$C>0$ such that for all small $r>0$
\begin{align}
  \label{eq:Lebergodic}
  \limsup_{T\nearrow\infty}\frac1T\int_0^T\Delta_r\big(\phi_t(x)\big) \, dt
  \le
  C\int_{\|u\|_2<r}\big((\log \|u\|_2)^2-(\log r)^2\big) d\lambda_2(u);
\end{align}
where $\lambda_2$ is the Lebesgue measure on
$\RR^2$. 
It is easy to see that the above expression tends to zero when
$r\to0+$, as we need.

\subsubsection*{Reduction to plane dynamics}

From Remark~\ref{rmk:symmetry}, each trajectory crossing $\cC$ is
contained in one vertical plane through the $z$-axis. We can assume,
without loss of generality, that we are dealing with a flow like $Y_0$,
whose trajectories are depicted in Figure~\ref{fig:H0}, to
estimate the value of the integral in~\eqref{eq:SSR}.

Considering $0<r\ll 1$, then trajectories outside of $\cC$ do not
contribute to the above integral --- we consider only those entering
$\cC$ through a small neighborhood $I_0=(-r,r)\times\{-2\}$ on
$\Sigma_{-2}$. We assume (without loss of generality) that from $I_0$
to the ball $B_r(\sigma_1)$ the flow is essentially tubular: starting
at $(x_0,-2)$ we will arrive at $B_r(\sigma_1)$ with the first
coordinate still equal to $x_0$. Likewise, between $B_r(\sigma_1)$ and
$B_r(\sigma_2)$ and from $B_r(\sigma_2)$ and $(x_0,2)$ we assume that
the flow is tubular.

From~\cite[Theorem 1.3]{Newhouse16} (here we need $Y_3$ to be at least
$C^{1+\alpha}$ for some $0<\alpha<1$), we can locally $C^1$ linearize
the flow around $\sigma_1$ in the ball $B_r(\sigma_1)$ (reducing the
value of $r>0$ if needed): there exists a $C^1$ diffeomorphism
$\zeta:B_r(\sigma_1)\to\RR^2$ so that
$\zeta(\phi_t(w))=e^{Dt}\zeta(w)$ for $w\in B_r(\sigma_1), t>0$ so
that $\phi_{[0,t]}(w)\subset B_r(0)$ and $D=\diag\{1/5,-1/5\}$.

Therefore, the distance $d(\phi_t(w),\sigma_1)$ can be estimated by
$\|e^{tD}\zeta(w)\|_2$ in the Euclidean norm, and so the integral
in~\eqref{eq:SSR} for a trajectory starting at the boundary of
$B_r(\sigma_1)$ can be calculated, writing $\zeta(w)=(x_0,r)$ with
$x_0\neq0$
\begin{align*}
  \int_0^t-\frac12\log\|e^{tD}\zeta(w)\|_2^2\,dt
  &=
  -\frac12\int_0^t\log(e^{2s/5}x_0^2+e^{-2s/5}r^2)\,ds
  \le
  -\frac12\int_0^t\log(e^{2s/5}x_0^2)\,ds
  \\
  &=
  -\int_0^t(s/5+\log |x_0|)\,ds
  =
  -t(t/10+\log |x_0|).
\end{align*}
The trajectory leaves $B_r(\sigma_1)$ before the time $t_0$ so that
$e^{t_0/5}|x_0|=r \iff t_0=5\log(r/|x_0|)$. Thus each trajectory crossing
$B_r(\sigma_1)$ contributes to the integral in~\eqref{eq:SSR} by at
most
$S=-t_0(t_0/10+\log |x_0|)=\frac52\big( (\log |x_0|)^2-(\log r)^2\big)$.

The second coordinate of $\zeta(\phi_{t_0}(w))$ at the exit from
$B_r(\sigma_1)$ is $e^{-t_0/5}r=x_0$ again. From $B_r(\sigma_1)$ to
$B_r(\sigma_2)$ we can likewise assume that the flow is tubular, and
repeat the calculation again when crossing $B_r(\sigma_2)$.

We thus obtain that at each crossing of $\cC$ starting at $(x_0,-2)$
we arrive at $(x_0,2)$ after a time $\tau(x_0,-2)$ and for some
constant $C>0$
\begin{align}
  \label{eq:S1}
  \int_0^{\tau(x_0,-2)}\Delta_r\big(\phi_t(w)\big)\,dt
  \le
  C\cdot (2 S) = 5 C \big( (\log |x_0|)^2-(\log r)^2\big).
\end{align}

\subsubsection*{Back to the dynamics on $M$}

We now consider a trajectory starting at $\leb$-generic point
$w\in\Sigma_0^*$ and crossing $\cC$ through $I_0$ at times
$t_n<T_n<t_{n+1}$ so that $P^{k_n}w=\phi_{t_n}w\in I_0$ and
$T_n=t_n+\tau(P^{k_n}w)$, where $P^{k_i}w$ are precisely those
iterates which fall in $I_0$.  At every visit to $B_r(z)\times{0}$ on
$N\times\{0\}$ the expression $|x_0|$ in the upper bound
from~\eqref{eq:S1} means $d(P^{k_n}w,z)$.  We can estimate as follows
\begin{align*}
  \int_0^{T_n}\Delta_r\big(\phi_s(w)\big)\,dt
  &=
    \sum\nolimits_{i=1}^n\int_{t_i}^{T_i}
    \Delta_r\big(\phi_s(w)\big)\,dt
  \\
  &\le
  5C\sum\nolimits_{i=1}^{k_n}
    \big((\log d(P^iw,z))^2-(\log r)^2\big)
    \cdot\mathbf{1}_{B_r(z)}(P^iw).
\end{align*}
Thus, we can estimate the average as
\begin{align*}
  \frac1{T_n}\int_0^{T_n}\Delta_r\big(\phi_s(w)\big)\,dt
  &\le
    \frac{5C\cdot k_n}{T_n}\cdot\frac1{k_n}
    \sum\nolimits_{i=1}^{k_n}
    \big((\log d(P^iw,z))^2-(\log r)^2\big)
    \cdot\mathbf{1}_{B_r(z)}(P^iw);
\end{align*}
but we also have $T_n=\sum_{i=0}^{k_n}\tau(P^iw)$ (recall the
definition of $\tau$ as the Poincar\'e time associated to the
Poincar\'e first return map) so that we
can use  the $P$-invariance and ergodicity of $\leb_\Sigma$ to get
\begin{align*}
  \limsup_{n\to\infty}
  \frac1{T_n}\int_0^{T_n}\Delta_r\big(\phi_s(w)\big)\,dt
  &\le
    \frac{5C}{\leb_\Sigma(\tau)}\int_{B_r(z)}
    \big( (\log d(w,z))^2-(\log r)^2\big)\,d\leb_\Sigma(w)
  \\
  &=
    \frac{5C}{\leb_\Sigma(\tau)}
    \int_{u\in B_r(0)\subset\RR^2}\big((\log \|u\|_2)^2-(\log
    r)^2\big)\,d\lambda_2(u),
\end{align*}
where $\lambda_2$ is the Lebesgue area measure on the Euclidean plane.

Given any strictly increasing and unbounded positive real sequence
$(s_m)_{m\ge1}$, we have the following two cases
\begin{description}
\item[$T_{n_m}<s_m<t_{n_m+1}$] we get the bound
  \begin{align*}
    \frac1{s_m}\int_0^{s_m}\Delta_r\big(\phi_t(w)\big)\,dt
    =
    \left(\frac{T_{n_m}}{s_m}\right)\cdot\frac1{T_{n_m}}
    \int_0^{T_{n_m}}\Delta_r\big(\phi_t(w)\big)\,dt
    \le
    \frac1{T_{n_m}}
    \int_0^{T_{n_m}}\Delta_r\big(\phi_t(w)\big)\,dt; 
  \end{align*}
\item[$t_{n_m}\le s_n<T_{n_m}$] we get the bound
    \begin{align*}
      \frac1{s_m}\int_0^{s_m}\Delta_r\big(\phi_t(w)\big)\,dt
    \le
    \left(\frac{T_{n_m}}{s_m}\right)\cdot\frac1{T_{n_m}}
    \int_0^{T_{n_m}}\Delta_r\big(\phi_t(w)\big)\,dt.
    \end{align*}
  \end{description}
Since $T_n=t_n+\tau(P^{k_m}w)>s_m\ge t_m$ then
\begin{align}\label{eq:quotient}
  \frac{T_{n_m}}{s_m}
  \le
  \frac{t_{n_m}+\tau(P^{k_m}w)}{t_{n_m}}
  =
  1+\frac{\tau(P^{k_m}w)}{t_{n_m}}
  =
  1+\frac{\tau(P^{k_m}w)/k_m}{\frac1{k_m}\sum_{i=0}^{k_m-1}\tau(P^iw)}.
\end{align}
For $\leb_\Sigma$-a.e. $w$, from $P$-invariance and ergodicity we have
\begin{align*}
  \frac1{k_m}\sum_{i=0}^{k_m-1}\tau(P^iw) \to \leb_\Sigma(\tau)
  \qand
  \frac{\tau(P^{k_m}w)}{k_m}\to 0;
\end{align*}
so that~\eqref{eq:quotient} tends to $1$ for large $m$.
Altogether, this shows that
\begin{align*}
  \limsup_{n\to\infty}
  \frac1{s_m}\int_0^{s_m}\Delta_r\big(\phi_t(w)\big)\,dt
  =
\limsup_{n\to\infty}
  \frac1{T_n}\int_0^{T_n}\Delta_r\big(\phi_s(w)\big)\,dt;
\end{align*}
completing the proof of~\eqref{eq:Lebergodic}.

\subsection{Higher codimensional NU2SE with mixed
  sectional hyperbolic equilibria}
\label{sec:section-hyperb-equil}

We repeat the construction starting with the vector field $Y_1$
from~Subsection~\ref{sec:multid-ash-attract}, that is, we consider
\begin{align*}
  Y_4(\varrho \cos\theta, \varrho \sin\theta, z)
  :=
  H^\prime_y(\varrho,z)\cdot(\cos\theta,\sin\theta,0)-
  2\cdot H^\prime_x(\varrho,z)\cdot X.
\end{align*}
We note that the action of the flow $\phi_t$ of $Y_4$ on $2$-planes is
given by the additive compound $\wedge^{[2]}DY_4$: i.e. given that
$D\phi_t(w)$ is the solution of the Linear Variational Equation on
$\RR^3$
\begin{align*}
  Z'=DY_4(\phi_t(w))\cdot Z, \qquad Z_0=I_3:\RR^3\to\RR^3;
\end{align*}
then $\wedge^2D\phi_t(w)$ is the solution of
\begin{align*}
Z'=\wedge^{[2]}DY_4(\phi_t(w))\cdot Z, \qquad Z_0=I_3:\RR^3\to\RR^3;
\end{align*}
and we can use the following
\begin{align*}
  A =
  \begin{pmatrix} a_{11} & a_{12} & a_{13} \\ a_{21} & a_{22} &
  a_{23} \\ a_{31} & a_{32} & a_{33} \end{pmatrix}
\implies
\wedge^{[2]}A = \begin{pmatrix}
a_{11} + a_{22} & a_{23} & -a_{13} \\
a_{32} & a_{11} + a_{33} & a_{12} \\
-a_{31} & a_{21} & a_{22} + a_{33}
\end{pmatrix}                  
\end{align*}
(see e.g.~\cite{Muldowney1990} or~\cite{fiedler74,Zhang2013} for short
introductions, where $\wedge^{[2]}A$ is written $A^{[2]}$)
so that for $\varrho\le2$ we get $\diver(Y_4)\equiv 0$ and
\begin{align*}
  DY_4(x,y,z)
  =
  \begin{pmatrix}
    -z/5 & 0 & -x/5
    \\
    0 & -z/5 & -y/5
    \\
    6x/5&6y/5& 2z/5
  \end{pmatrix}                 
               \,\&\,
               \wedge^{[2]}DY_4(x,y,z)=
  \begin{pmatrix}
    -2z/5 & -y/5 & x/5
    \\
    6y/5 & z/5 &0
    \\
    -6x/5 &0& z/5
  \end{pmatrix}.
\end{align*}
Therefore, the hyperbolic equilibria $\sigma_{1,2}$ became
sectional hyperbolic
\begin{align*}
  DY_4(0,0,\pm1)
  =
  \pm
  \begin{pmatrix}
    -1/5 & 0 & 0
    \\
    0 & -1/5 & 0
    \\
    0 & 0  & 2/5
  \end{pmatrix}                 
               \,\&\,
               \wedge^{[2]}DY_4(0,0,\pm1)=\pm
  \begin{pmatrix}
    -2/5 & 0&0
    \\
    0 & 1/5 &0
    \\
    0 &0& 1/5
  \end{pmatrix};
\end{align*}
with $\sigma_1$ generalized Rovella-like and $\sigma_2$ generalized
Lorenz-like.

\subsubsection{Asymptotical sectional-expansion}
\label{sec:asympt-section-expan-4}

From the arguments of Subsection~\ref{sec:poincare-transit-map-1},
since $Y_4$ is based on a symmetrization of $Y_2$, we recover
Claim~\ref{cl:trId} so that the Poincar\'e transition map on
$\Sigma_0^*$ is again the identity.

The upper bound from~\eqref{eq:highest} is kept with
$\|DY_4\|= \sup_{w\in \cC}\|DG_4(w)\|$ in the place of $\|DY_3\|$; and
so the same argument from Subsection~\ref{sec:asympt-section-expan-3}
shows that the flow of $G_4$ --- obtained from $X$ by attaching $Y_4$
as in~\eqref{eq:attached} --- satisfies NU2SE as long as $U$ is small
enough and thus, from Remark~\ref{rmk:NUSE-wNU2SE}(2), the attracting
set is ASH.

\subsubsection{Slow recurrence to equilibria and physical measure}
\label{sec:slow-recurrence-}

The same argument of~\ref{sec:slow-recurr-equilibr} applies here, with
similar upper bound, to conclude the SR condition on a full volume
subset of $M$, whenever $Y_3$ is of class $C^2$, and so the attracting
set becomes MNUSE. Again, from Theorem~\ref{thm:discretefabv} we
conclude that there exists a unique physical/SRB measure whose basin
covers $\leb$-a.e. point of the ambient space.



\def\cprime{$'$}

 \bibliographystyle{abbrv}

\end{document}